\documentclass{article}
\usepackage{lineno,hyperref}
\usepackage{amsfonts}
\usepackage{amsmath}
\usepackage{latexsym,amssymb,color}
\usepackage{amsmath}
\usepackage{indentfirst}
\usepackage{tikz}
\usepackage{euscript}
\usepackage[all]{xy}
\usepackage{empheq}
\usepackage{fancybox}
\usepackage{framed}
\usepackage[square,sort,comma,numbers]{natbib}
\usepackage{bibentry}
\usepackage{comment}

\newtheorem{theorem}{Theorem}[section]
\newtheorem{lemma}[theorem]{Lemma}
\newtheorem{example}[theorem]{Example}
\newtheorem{remark}[theorem]{Remark}

\newtheorem{definition}[theorem]{Definition}

\usepackage{color}
\input xypic
\input xy
\xyoption{v2}
\xyoption{all}
\xyoption{2cell}
\makeatother
\newcommand{\pmor}[2] {#1 \xrightarrow{ \ #2\  } \left(#1+\partial_{1}(#2)\right)}
\newcommand{\pmoru}[2] {\s{#1} \xrightarrow{ \ \s{#2} \ } \left(\s{#1+\partial_{1}(#2)}\right)}

\newcommand{\pmoo}[3] {#1 \xrightarrow{ \ #2 \  } #3}

\newcommand{\pmoruu}[3] {\s{#1} \xrightarrow{ \  \s{#2} \ } \left(\s{#3}\right)}

\newcommand{\trinearly}[4] {& & \s{#4 +\partial_{1}(#3) +\partial_{1}(#2)}\\ & & \quad \quad \quad \\
\s{#4}  \ar[rruu]^{\s{#3+#2}}\ar[rr]_{\s{#3}} & & {\s{#4}+\partial_{1}(#3)} \ar[uu]_{\s{#2 +\partial_{2}(#1)}}^{\ovalbox{$\s{#1}$}}}
\newcommand{\trinearlyprime}[4] {& & \s{#4 +\partial'_{1}(#3) +\partial'_{1}(#2)}\\ & & \quad \quad \quad \\
	\s{#4}  \ar[rruu]^{\s{#3+#2}}\ar[rr]_{\s{#3}} & & {\s{#4}+\partial'_{1}(#3)} \ar[uu]_{\s{#2 +\partial'_{2}(#1)}}^{\ovalbox{$\s{#1}$}}}
\newcommand{\trinearlyn}[4] {& & #4+\partial_{1}(#3)+\partial_{1}(#2)\\ & & \quad \quad \quad \\
#4  \ar[rruu]^{#3+#2}\ar[rr]_{#3} & & {#4+\partial_{1}(#3)} \ar[uu]_{\s{#2 +\partial_{2}(#1)}}^{\ovalbox{$#1$}}}
\newcommand{\triangl}[4] {& & \s{#1+\partial_{1}(#2)+\partial_{1}(#3)}\\ & & \quad \quad \quad \\
\s{#1}  \ar[rruu]^{\s{#2+#3}}\ar[rr]_{#2} & & \s{{#1}+\partial_{1}(#2)} \ar[uu]_{\s{#3+\partial_{2}(#4)}}^{\ovalbox{$\s{#4}$}}}
\newcommand{\trinearlyy}[7] {& & \s{#4} \\ & & \quad \quad \quad \\
\s{#2}  \ar[rruu]^{\s{#7}}\ar[rr]_{\s{#5}} & & {\s{#3}}  \ar[uu]_{\s{#6}}^{\ovalbox{$\s{#1}$}}}
\newcommand{\triangly}[7] {& & \s{#3} \\ & & \quad \quad \quad \\
\s{#1}  \ar[rruu]^{\s{#6}} \ar[rr]_{\s{#4}} & & {\s{#2}}  \ar[uu]_{\s{#5}}^{\ovalbox{$\s{#7}$}}}

\newcommand{\triangln}[7] {& \s{#3} & \\
\s{#1} \ar[ru]^{\s{#6}} \ar[rr]_{\s{#4}}^{\ovalbox{$\s{#7}$}} & & \s{#2} \ar[lu]_{\s{#5}}}
\newcommand{\twosimplex}[3] {& & \s{#3} \\ \\
\s{#1}  \ar[rruu] \ar[rr] & & {\s{#2}} \ar[uu]}
\newcommand{\threesimplex}[4] { & \s{#3} & \\ & \s{#4} \ar[u] & \\ \s{#1} \ar[rr] \ar[ru] \ar[ruu] & & \s{#2} \ar[lu] \ar[luu] }
\newcommand{\free}[5]{ \s{#1} \ \ar@{^{(}->}[r]^{\s{incl}} \ar[dr]_{\s{#4}} & \s{#2} \ar@{.>}[d]^{\s{#5}}_{\s{\exists !}} \\ & \s{#3}}


\newcommand{\tetraprime}[7] {
	& & \s{#1+\partial'_{1}(#2)+\partial'_{1}(#3)+\partial'_{1}(#5)} & & \\
	& & & & \\
	& & \s{#1+\partial'_{1}(#2)+\partial'_{1}(#3)} \ar[uu]|{\s{#5+\partial'_{2}(#6)+\partial'_{2}(#7)}} & & \\
	& & & & \\
	\s{#1} \ar[rrrr]_{\s{#2}}^{\ovalbox{$\s{#4}$}} \ar[rruu]_{\s{#2+#3}} \ar[rruuuu]^{\s{#2+#3+#5}\quad}_{\ovalbox{$\s{#6+#7}$}} & & & & \s{#1+\partial'_{1}(#2)} \ar[lluu]^{\s{#3+\partial'_{2}(#4)}} \ar[lluuuu]_{\quad \quad \s{#3+#5+\partial'_{2}(#4)+\partial'_{2}(#6)}}^{\ovalbox{$\s{#7}$} \quad }
}

\newcommand{\tetrn}[7] {
& & #1+\partial_{1}(#2)+\partial_{1}(#3)+\partial_{1}(#5) & & \\
& & & & \\
& & #1+\partial_{1}(#2)+\partial_{1}(#3) \ar[uu]|{#5+\partial_{2}(#6)+\partial_{2}(#7)} & & \\
#1 \ar[rrrr]_{#2}^{\ovalbox{$#4$}} \ar[rru]_{#2+#3} \ar[rruuu]^{#2+#3+#5}_{\ovalbox{$#6+#7$}} & & & & #1+\partial_{1}(#2) \ar[llu]^{#3+\partial_{2}(#4)} \ar[lluuu]_{\ \ \ #3+#5+\partial_{2}(#4)+\partial_{2}(#6)}^{\ovalbox{$#7$}}
}
\makeatletter

\newenvironment{idea}[1][Idea]{\textbf{#1.} }

\newenvironment{problem}[1][Problem]{\textbf{#1:} }{\ \rule{0.5em}{0.5em}}
\newenvironment{proof}[1][\emph{Proof}]{\textbf{#1:} }{\ \rule{0.5em}{0.5em}}

\newtheorem{warning}[theorem]{Warning}

\def \k {\kappa}
\def \d {\partial}
\def \s {\scriptstyle}
\def \A {{\cal A}}
\def \B {{\cal B}}

\def \t {\blacktriangleright}
\def \G {\mathcal{G}}
\def \ra{\xrightarrow}
\allowdisplaybreaks

\begin{document}

\title{Pointed Homotopy of Maps Between 2-Crossed Modules of Commutative Algebras}

\author{\.{I}.\.{I}lker Ak\c{c}a \\ {\it \small i.ilkerakca@gmail.com}\\
	{   \small Department of Mathematics and Computer Science},\\ { \small Eski\c{s}ehir Osmangazi University, Turkey.}
	\\
	\quad \\
	Kad\.{i}r Em\.{i}r\thanks{KE expresses his gratitude for the hospitality of CMA in 2014/15.} \\ {\it  \small k.emir@campus.fct.unl.pt} \\ { \small Departamento de Matem\'atica,}\\ { \small  Faculdade de Ci\^encias e Tecnologia}{ \small (Universidade Nova de Lisboa),} 
	\\{  \small  Quinta da Torre, 2829-516 Caparica,  Portugal.} \\
	{  \small Also at: Department of Mathematics and Computer Science,}\\ { \small Eski\c{s}ehir Osmangazi University, Turkey.} \\ \quad \\ Jo\~ao Faria Martins\thanks{JFM was partially supported by CMA/FCT/UNL, under the project UID/MAT/00297/2013 of FCT and by FCT(Portugal) through the  ``Geometry and Mathematical Physics Project'',
		FCT EXCL/MAT-GEO/0222/2012.} \\{\it  \small jn.martins@fct.unl.pt}\\
	{ \small Departamento de Matem\'atica and Centro de Matem\'atica e Aplica\c{c}\~oes,}\\{ \small Faculdade de Ci\^encias e Tecnologia (Universidade Nova de Lisboa),} \\ { \small Quinta da Torre,
		2829-516 Caparica, Portugal.}}

\maketitle

\begin{abstract}
	{We address the homotopy theory of 2-crossed modules of commutative algebras, which are equivalent to simplicial commutative algebras with Moore complex of length two. {In particular, we construct for maps of 2-crossed modules a homotopy relation, and prove that it yields an equivalence relation in very unrestricted cases (freeness up to order one of the domain 2-crossed module).}  This latter condition strictly includes the case when the domain is cofibrant.  Furthermore, we prove that this notion of homotopy yields a groupoid with objects being the 2-crossed module maps between two fixed 2-crossed modules (with free up to order one domain),  the morphisms being the homotopies between 2-crossed module maps.}
\end{abstract}


\noindent{\bf 2010 AMS Classification:} 
{55U10 (principal), 
{18D05,    
18D20,  
55Q15    
(secondary).}

\section*{Introduction}
{A crossed module \cite{BHS} $\G=(\d\colon E \to G, \t)$, of groups, is given by a group map $\d\colon E \to G$, together with an action $\t$ of $G$ on $E$, by automorphisms, such that the Peiffer-Whitehead relations, below, hold for each $e,f \in E$ and each $g \in G$:}
\begin{align*}
 \textrm{ \bf{First Peiffer-Whitehead Relation (for groups)}: }\, \d(g \t e)& =g \,\d(e) \, g^{-1},&&&&& &&&&&\\  \textrm{\bf{ Second Peiffer-Whitehead Relation (for groups)}: }\, \d(e)  \t f&=e\, f \, e^{-1}. &&&&& &&&&&
\end{align*}
{Group crossed modules were firstly introduced by Whitehead in \cite{W3,W4}.} They are  algebraic models for homotopy 2-types, in the sense that \cite{B1,Lo} the homotopy category of the model category \cite{CG,BG2} of group crossed modules  is equivalent to the homotopy category of the model category \cite{EHL} of pointed 2-types: pointed connected spaces whose homotopy groups $\pi_i$ vanish, if $i\ge 3$.  {Crossed modules of groups also naturally appear in the context of simplicial homotopy theory, namely they are  equivalent to simplicial groups with Moore complex of length one \cite{C1} and analogously for crossed modules of groupoids  \cite{MuPu2}.}
The homotopy relation between crossed module maps $\G \to \G'$ was introduced by Whitehead in \cite{W4}, in the {context of ``homotopy systems'', now called free crossed complexes. In \cite{BH1} (see also  \cite{BHS}), homotopy was investigated in terms of a monoidal closed structure on crossed complexes, and an interval object. Homotopy for crossed complexes was also developed by Huebschmann in \cite{Hueb}.}

{The homotopy relation between crossed module maps $\G \to \G'$ can be equivalently addressed either by considering natural functorial path objects for $\G'$ or cylinder objects  for $\G$.} It yields, given any two crossed modules $\G$ and $\G'$, a groupoid of maps $\G \to \G'$ and their homotopies. In particular, the homotopy relation between crossed module maps $\G \to \G'$ is an equivalence relation in the general case, with no restriction on $\G$ or $\G'$. This should be {compared} with what would be guaranteed from the model category \cite{DS} point of view, where  we would expect homotopy of maps $\G \to \G'$ to be an equivalence relation only when  $\G =(\d\colon E \to G, \t)$ is cofibrant {(given 
that any object is fibrant). In the well-known model category structure in the category of 
crossed modules  \cite{CG}, obtained by transporting the usual model 
category structure of the category of 
simplicial sets,  $\G =(\d\colon E \to G, \t)$  is cofibrant if, and only if,  $G$ is  a free group (\cite{N1}).

The notion of a 2-crossed module of groups was addressed by Conduch\'{e} in \cite{C1}. A 2-crossed module of groups $\A=(L \ra{\delta} E \ra{\d} G, \t,\{,\})$ is given by a complex of groups, together with actions $\t$ of $G$ on $L$ and $E$, making it a complex of $G$-modules, where $G$ acts on itself by conjugation. Looking at the map $\d\colon E \to G$ and the action of $G$ on $E$, the first {Peiffer-Whitehead relation is automatically satisfied. However the second is not, in general, and thus $(\d\colon E \to G,\t)$ is what is called a pre-crossed module; see \cite{BHS}.} We have a map $(e,f) \in E \times E \longmapsto \langle e,f \rangle \doteq efe^{-1} \,\, \d(e) \t f^{-1} \in E, $ called the Peiffer pairing, and the map $\{,\}\colon E \times E \to L$, called the Peiffer lifting, should be a lifting of it to $L$,  satisfying itself a number of very natural properties. 

Conduch\'{e} proved in \cite{C1} that the category of 2-crossed modules of groups is equivalent to the category of simplicial groups with 
Moore 
complex of length two. This can be used to 
prove that the homotopy category of 2-crossed modules is equivalent to the homotopy category of homotopy 3-types: pointed connected spaces $X$ such $\pi_i(X)=0$, if $i \ge 3$. {This latter result was directly proven in \cite{M1}, where a homotopy relation for 2-crossed module maps was defined. Moreover, a fundamental 2-crossed module was associated to each pointed CW-complex, and proven to have good geometric realisation properties at the level of 2-crossed module maps and their homotopies.}

{Also well known is the fact that a 2-crossed module of groups uniquely represents a Gray-category, with a single object; see, for example, \cite{BG1}. In fact, the notion of homotopy between 2-crossed module maps \cite{GM1,M1} is equivalent to the notion of a 1-transfor (pseudo-natural transformation) between Gray functors \cite{Crans}. Thus 2-crossed module homotopy is not equivalent to the more cumbersome definition of simplicial homotopy of simplicial groups with Moore complex of length two; \cite{GMO}.}

{Another important algebraic model for  homotopy 3-types is provided by the crossed squares \cite{Lo}, equivalent to cat$^2$-groups. They are closely related to bisimplicial groups. As a model for homotopy 3-types, crossed squares have the major advantage that there exists a fundamental crossed square functor from the category of triads of pointed spaces to the category of crossed squares, under mild conditions  preserving colimits \cite{BLo}. This fundamental crossed square can be defined out of the usual relative and triad homotopy groups, considering also Whitehead products. Pointed CW-complexes naturally give rise to triads \cite{Ellis}.  Applying the fundamental crossed square functor to such a triad  gives rise to a crossed square retaining all of the homotopy information of the CW-complex up to degree three.  We therefore have a concrete and natural way to associate a crossed square to a CW-complex, representing its homotopy 3-type. This concreteness is in sharp contrast with the case of quadratic modules, developed in \cite{B1}. The notion of crossed square homotopy has not been developed, on the contrary to homotopy of quadratic modules and quadratic chain-complexes, addressed in \cite{B1}.}

A 2-crossed module $\A=(L \ra{\delta} E \ra{\d} G, \t,\{,\})$, of groups, is called free up to order one if $G$ is a free group, $\A$ being called free up to order two if, furthermore, $(\d\colon E \to G,\t)$ is a free pre-crossed module; \cite{GM1,M1}.
In the model category of 2-crossed modules of groups, \cite{CG1,CG,GM1,La}, a 2-crossed module $\big(L \ra{\delta} E \ra{\partial} G, \t,\{,\}\big)$ is cofibrant if, and only if, it is a retract of a free up to order two 2-crossed module, any 2-crossed module being fibrant. What is surprising, and was proved in \cite{GM1,G1}, is that if $\A$ is solely free  up to order one (and, therefore, in general, far from being cofibrant), and $\B$ is any other 2-crossed module, then we have a groupoid of 2-crossed module maps $\A \to \B$ and their homotopies, the latter being constructed from a functorial path-space. This groupoid can be upgraded to a 2-groupoid by considering 2-fold homotopies between homotopies. {We note that the assumption that the domain 2-crossed module is free up to order one is strictly necessary, in order that we can compose and invert homotopies. This was explicitly proved in \cite{GM1}. }

All algebras in this paper are taken to be commutative. In \cite{DGV1,P1}, we can find the definition of crossed modules and  of 2-crossed modules of algebras, the latter having the form $\A=(L\overset{\partial _{2}}{\longrightarrow }E\overset{\partial _{1}}{ \longrightarrow } R,\t,\{,\})$, where $L$, $E$ and $R$ are algebras, and all the rest parallels the group case, essentially switching actions by automorphisms to actions by multipliers. {It is proven in \cite{A1,GV1,P1} that simplicial  algebras with Moore complex of length two correspond, through taking Moore complexes, to 2-crossed modules of commutative algebras.}  Moreover, we have an inclusion functor from the category of 2-crossed modules of algebras into the category of simplicial algebras, {and a reflection functor making the former a reflexive subcategory;} \cite{AP3}. 

By the general construction in \cite[II,5]{GJ1} and \cite{CG1} (which follow a similar pattern   as the one  appearing in \cite{Q1}), we have a Quillen model structure in the category of 2-crossed modules of commutative algebras, obtained by transporting the usual structure in the category of simplicial sets, and a built in Quillen pair: 
\begin{equation*}{\hskip-0.5cm\xymatrix@C=5pt{&\{\textrm{Simplicial Sets}\} \ar@/^1.0pc/[r]^{F} & \quad \quad \quad \quad \quad  \{2-\textrm{Crossed Modules of Commutative Algebras}\}\ar@/^1.0pc/[l]^{U}}.}\end{equation*}
Here $U$ is the forgetful functor from the category of simplicial commutative algebras to the category 
 of simplicial sets (recalling that a 2-crossed module of commutative algebras is essentially a simplicial commutative algebra with Moore complex of length two). On the other hand, in the opposite direction, one uses the free algebra (on  a set) functor (the polynomial algebra with one symbol for each element of the set), followed by the reflection from the category of simplicial algebras onto the category of 2-crossed modules of algebras. This approach can be used to prove that  a 2-crossed module of algebras is cofibrant if, and only if, it is a retract of a 2-crossed module $\A=(L\overset{\partial _{2}}{\longrightarrow }E\overset{\partial _{1}}{ \longrightarrow } R,\t,\{,\})$  which is free up to order two. {The latter means that the  algebra pre-crossed module $(\d_1\colon E \to R,\t)$ is free (see \cite{AP3})  and so is $R$. } 
 
In this paper, we address the homotopy theory of maps between crossed modules and 2-crossed modules of commutative algebras. In particular, we will prove that, if $\G$ and $\G'$ are crossed modules of algebras, without any restriction on $\G$ or $\G'$, then we have a groupoid of crossed module maps $\G \to \G'$ and their homotopies, similarly to the group case. {As for the case of 2-crossed modules, the homotopy relation for 2-crossed module maps $\A \to \A'$ is not an equivalence relation. However, this feature can be corrected by restricting to the case when the domain $\A$ is free up to order one, and therefore in  much more generality than what would be derived from a model category \cite{DS} point of view. This would {give} that homotopy is an equivalence relation  when $\A$ is cofibrant (free up to order two), since all objects are fibrant.}

{A large chunk of the paper will be devoted to carefully  describing  the homotopy relation between 2-crossed module maps $\A \to \A'$, which will {lead} us to consider quadratic derivations, also considered in \cite{B1,GM1,M1}, in the group case. Also requiring detailed calculations, occupying the biggest part of this paper, and {being our main result},  is the proof that, given two 2-crossed modules $\A$ and $\A'$, with $\A=(L\overset{\partial _{2}}{\longrightarrow }E\overset{\partial _{1}}{ \longrightarrow } R,\t,\{,\})$, free up to order one (and with a chosen algebra basis $B$ of $R$),} then we have a groupoid of 2-crossed module maps $\A \to \A'$, and their homotopies. The construction of the composition of homotopies, and also the proof that this composition is associative, admitting inverses, requires appealing to several lateral, however important, constructions, such as the algebras of edges, triangles and tetrahedra in a 2-crossed module, paralleling the main arguments in \cite{GM1,G1}. {(This bit is highly technical, forcing us to introduce several auxiliary algebra actions, and it is where  the 
most difficult and important 
calculations in this paper live; see \ref{aaa}).} In particular, we have the result that the homotopy relation between 2-crossed module maps $\A \to \A'$ is an equivalence relation, if $\A$ is free up to order one.

\section{Preliminaries}

\noindent All algebras will be commutative and defined over $\kappa$, a fixed commutative ring.
\subsection{Pre-crossed modules and crossed modules}
\begin{definition}\label{action}
If $M$ and $R$ are (commutative) $\k$-algebras, a $\k$-bilinear map:
\begin{equation*}
(r,m) \in R\times M \longmapsto  r\blacktriangleright m \in M
\end{equation*}%
is called an action of $R$ on $M$ if, for all $m,m^{\prime }\in M$ and $r,r^{\prime }\in R$:
\begin{description}
\item[A1.]\label{a1} $r\blacktriangleright (mm^{\prime })=(r\blacktriangleright
m)m^{\prime }=m(r\blacktriangleright m^{\prime })$,
\item[A2.]\label{a2} $(rr^{\prime })\blacktriangleright m=r\blacktriangleright
(r^{\prime }\blacktriangleright m).$
\end{description}
\end{definition}
\begin{definition}\label{pcm}
A pre-crossed module of $\k$-algebras $(E,R,\partial )$ is given by a homomorphism of $\k$-algebras $\partial\colon E \to R$, together with a left action $\blacktriangleright $ of $R$ on $E$, such that
the following relation, called \textquotedblleft first {Peiffer-Whitehead} relation\textquotedblright, holds:
\begin{description}
\item[XM1)] $\partial (r\blacktriangleright e)=r\,\partial (e)$, for each $e \in E$ and $r \in R$.
\end{description}
A crossed module of commutative $\k$-algebras $(E,R,\partial )$ is a pre-crossed
module satisfying, furthermore \textquotedblleft the second {Peiffer-Whitehead} relation\textquotedblright:
\begin{description}
\item[XM2)] $\partial (e)\blacktriangleright e^{\prime }=ee^{\prime }$, for all $e,e^{\prime }\in E$.
\end{description}
\end{definition}

\begin{example}
Let $R$ be a $\k$-algebra and $E\trianglelefteq R$ be  an ideal of $R$. Then $(E,R,i)$, where $i\colon E \to R$ is the inclusion map, is a
crossed module, where we use the multiplication in $R$ to define the action of $R$ on $E$: 
$(r,e) \in R\times E \longmapsto  r\blacktriangleright e=re \in E.$
\end{example}
%

\begin{definition}
Let $(E,R,\partial )$ and $(E^{\prime },R^{\prime },\partial ^{\prime })$ be
two crossed modules. A crossed module morphism $f \colon (E,R,\partial ) \to (E^{\prime },R^{\prime },\partial ^{\prime })$ is a pair %
$f=(f_{1}\colon E \to E',f_{0}\colon R \to R')$ of algebra morphisms, making the diagram below commutative:
\begin{equation*} 
\xymatrix@R=20pt@C=60pt{
E
\ar[r]^{\partial}
\ar[d]_{f_{1}}
& R
\ar[d]^{f_{0}}
\\
E^{\prime}
\ar[r]_{{\partial}^\prime}
& R^{\prime}}\,
\end{equation*}
also preserving the action of $R$ on $E$:
$f_{1}(r\blacktriangleright e)=f_{0}(r)\blacktriangleright f_{1}(e) \textrm{ for all } e \in E,\, r \in R
$.\end{definition}
\subsection{2-crossed modules of commutative algebras}

\begin{definition}
{A 2-crossed module  $(L,E,R,\partial _{1},\partial _{2},\{,\})$ of (commutative) algebras is given by a chain complex of algebras:}
\begin{equation*}
L\overset{\partial _{2}}{\longrightarrow }E\overset{\partial _{1}}{%
\longrightarrow }R
\end{equation*}%
together with left actions $\blacktriangleright $ of $R$ on $E$ and $L$ (and
also on $R$ by multiplication), preserved by $\d_1$ and $\d_2$, and an $R$-bilinear function (called the Peiffer lifting):
\begin{equation*}
\{\ \ \otimes \ \ \}:E\otimes _{R}E\longrightarrow L,
\end{equation*} {denoted by $\{e,f\}$ or $\{e \otimes f\}$, where $e,f \in E$,} satisfying the following axioms, for all $l,l^{\prime }\in L,\ e,e^{\prime },e^{\prime \prime
}\in E$ and  $r\in R$:
\begin{description}
\item[2XM1)] $\partial _{2}\{e\otimes e^{\prime }\}=ee^{\prime }-\partial
_{1}(e^{\prime })\blacktriangleright e$,

\item[2XM2)] $\{\partial _{2}(l)\otimes \partial _{2}(l^{\prime
})\}=ll^{\prime }$,

\item[2XM3)] $\{e\otimes e^{\prime }e^{\prime \prime }\}=\{ee^{\prime
}\otimes e^{\prime \prime }\}+\partial _{1}(e^{\prime \prime
})\blacktriangleright \{e\otimes e^{\prime }\}$,

\item[2XM4)] $\{\partial _{2}(l)\otimes e\}=e\blacktriangleright ^{\prime
}l-\partial _{1}(e)\blacktriangleright l$,

\item[2XM5)] $\{e\otimes \partial _{2}(l)\}=e\blacktriangleright ^{\prime }l$,

\item[2XM6)] $r\blacktriangleright \{e\otimes e^{\prime
}\}=\{r\blacktriangleright e\otimes e^{\prime }\}=\{e\otimes
r\blacktriangleright e^{\prime }\}$.
\end{description}

\end{definition}

\begin{remark}
Note that \smash{$L\overset{\partial _{2}}{\longrightarrow }E$} is a crossed
module, with the  action of $E$ on $L$ being:%
\begin{equation*}
e\blacktriangleright ^{\prime }l=\{e\otimes \partial _{2}(l)\}
\end{equation*}%
However \smash{$E\overset{\partial _{1}}{\longrightarrow }R$} is (in general) only a
pre-crossed module. The image of the Peiffer lifting through $\d_2$ measures ({\bf{2XM1}}) the failure of $E\overset{\partial _{1}}{\longrightarrow }R$ to be a crossed module.
\end{remark}

\begin{definition}[Freeness up to order one]\label{freeness} Let $(L,E,R,\partial _{1},\partial _{2},\{,\})$ be a
2-crossed module. We say that this 2-crossed module is free up to order
one if $R$ is a free $\kappa$-algebra. In this paper, free up to order one 2-crossed modules will always come equipped with a specified {(free algebra) basis $B$ of $R$}, and, therefore, $R$ will be the algebra of polynomials over $\kappa$, with a formal variable assigned to each element of $B$.
\end{definition}

\begin{example}
Let $(E,R,\partial )$ be a pre-crossed module. Then
\smash{$\ker(\partial )\overset{i}{\longrightarrow }E\overset{\partial }{%
\longrightarrow }R$}, where $i\colon \ker (\d) \to E$ is the inclusion map,
is a 2-crossed module, with the Peiffer lifting:%
\begin{equation*}
\{\ \ \otimes \ \ \} \colon (e,e') \in  E\otimes _{R}E\longmapsto \{e\otimes e^{\prime }\}\doteq ee^\prime-\partial(e)e^\prime\in \ker(\partial).
\end{equation*}
\end{example}

\begin{definition}
Let $\A=(L,E,R,\partial _{1},\partial _{2},\{,\})$ and $\A'=(L^{\prime },E^{\prime
},R^{\prime },\partial _{1}^{\prime },\partial _{2}^{\prime },\{,\})$ be  2%
-crossed modules. A 2-crossed module map $f=(f_{2},f_{1},f_{0})\colon \A \to \A'$ is given by algebra maps $f_{0}\colon R\to R^{\prime }$,
$f_{1} \colon E\to E^{\prime }$ and $f_{2} \colon L\to L^{\prime}$,
making the diagram below
\begin{equation*} 
\xymatrix@R=20pt@C=60pt{
L
\ar[r]^{\partial _{2}}
\ar[d]^{f_{2}}
& E
\ar[r]^{\partial _{1}}
\ar[d]^{f_{1}}
& R
\ar[d]^{f_{0}}
\\
L^\prime
\ar[r]_{\partial _{2}^{\prime }}
& E^\prime
\ar[r]_{\partial _{1}^{\prime }}
& R^\prime
}
\end{equation*}
commutative and also preserving the actions of $R$ and $R'$, as well as the Peiffer liftings:
\begin{align*}
f_{1}(r\blacktriangleright e)&=f_{0}(r)\blacktriangleright
f_{1}(e),  \textrm{ for all }  e \in E  \textrm{ and } r\in R,
\\
f_{2}(r\blacktriangleright l)&=f_{0}(r)\blacktriangleright
f_{2}(l), \textrm{ for all }  l \in L  \textrm{ and } r\in R,
\\
f_{2}\{e\otimes e^{\prime }\}&=\{f_{1}(e)\otimes f_{1}(e^{\prime
})\},  \textrm{ for all } e,e^\prime\in E.
\end{align*}
\end{definition}

\section{The algebras of 0-, 1-, 2- and 3-simplices}

\noindent  Throughout this entire section, we fix a 2-crossed module $\A=(L,E,R,\partial _{1},\partial _{2},\{,\})$ of commutative algebras. Put $\A_0=R$, and call it the algebra of $0$-simplices in $\A$.

\subsection{Conventions on semidirect products}

\begin{definition}[Semidirect product]\label{semidirect} 
If we have an action $\blacktriangleright $ of  $R$
on a commutative algebra $E$, then our convention for the semidirect product $\k$-algebra $R\ltimes
_{\blacktriangleright }E$ is:
\begin{equation*}
(r,e)\, (r^{\prime },e^{\prime })=(rr^{\prime },r\blacktriangleright e^{\prime }+r^{\prime
}\blacktriangleright e+ee^{\prime }), \textrm { for all } e,e^{\prime }\in E \textrm { and }r,r^{\prime}\in R.
\end{equation*}%
\end{definition}
\noindent When considering actions of semidirect product algebras on other commutative algebras, the following lemma is very useful, and will be used without much comment:
\begin{lemma}\label{actionreduce}
Let $A$ be any commutative algebra. Then a bilinear map:
\begin{align*}
\left( (r,e),a\right) \in (R\ltimes _{\blacktriangleright }E)\times A\longmapsto   (r,e)\blacktriangleright a \in A
\end{align*}
is an algebra action if, and only if, for all $r,r' \in R$, $e,e' \in E$ and $a \in A$:
\begin{itemize}
 \item $\left[ (r,0)(r^{\prime },0)\right] \blacktriangleright a=(r,0)\blacktriangleright \left[
(r^{\prime },0)\blacktriangleright a\right] $,

 \item $\left[ (0,e)(0,e^{\prime })\right] \blacktriangleright a=(0,e)\blacktriangleright \left[
(0,e^{\prime })\blacktriangleright a\right] $,

 \item $(r,e)\blacktriangleright a=\left[ (r,0)+(0,e)\right] \blacktriangleright a$,
 
 \item $\left[ (r,0)(0,e^{\prime })\right] \blacktriangleright a=(r,0)\blacktriangleright \left[
(0,e^{\prime })\blacktriangleright a\right] $.
\end{itemize}
\end{lemma}

\subsection{The algebras of 1- and 2-simplices in a 2-crossed module}
\noindent The (commutative) algebra $\A_1\doteq R\ltimes
_{\blacktriangleright }E$ will be called the algebra of 1-simplices in $\A$.
It is convenient to express the elements $(r,e)\in \A_1$ in the following simplicial form:
\begin{equation}
\xymatrix{\pmor{r}{e}}.
\end{equation}
The product of two elements of  $\A_1$ can be ``visualized'' as:
\begin{multline*}
\big( \xymatrix{\pmor{r}{e}}\big) \cdot \big( \xymatrix{\pmor{r^%
\prime}{e^\prime}}\big)\\
=\big (
\xymatrix{\pmor{rr^\prime}{r\blacktriangleright e^{\prime }+r^{\prime
}\blacktriangleright e+ee^{\prime }}}\big).
\end{multline*}
\begin{remark}
{The reason for this notation is that there are non-trivial algebra maps:}
\begin{equation*}
d_{0}, d_{1}\colon \A_1=R\ltimes _{\blacktriangleright }E\longrightarrow R =\A_0,
\end{equation*}%
being:
\begin{align*}
&d_{0}\big( \xymatrix{\pmor{r}{e}}\big)  =  r, &
d_{1}\big( \xymatrix{\pmor{r}{e}}\big)  =  r+\partial_1 (e).
\end{align*}
There is also an algebra morphism:
\begin{equation*}
s_{0}:\A_0=R\longrightarrow R\ltimes _{\blacktriangleright }E=\A_1,
\end{equation*}%
being
$
s_{0}(r)=(r,0),$ {where $r \in R$}, which can be visualized as:
\begin{equation*}
s_{0}(r)=\big( \xymatrix{\pmor{r}{0}}\big) =  \big( \xymatrix{\pmoo{r}{0}{r}}\big). %
\end{equation*}
\end{remark}

\begin{lemma}\label{bullet}
There exists an action $\blacktriangleright _{\bullet
}$ of $(R\ltimes _{\blacktriangleright }E)$ on $(E\ltimes
_{\blacktriangleright ^{\prime }}L)$, having the form:%
\begin{equation*}
(r,e)\blacktriangleright _{\bullet }(e^{\prime },l)=(ee^{\prime }+r\blacktriangleright e^{\prime },\partial
_{1}(e)\blacktriangleright l+r\blacktriangleright l-\{e^{\prime }\otimes
e\}),
\end{equation*}%
for all $l\in L$, $e,e^{\prime }\in E$ and $r\in R$.
\end{lemma}

\begin{proof}
Note that:
\begin{align*}
& (r,0)\blacktriangleright _{\bullet }(e^{\prime },l)=(r\blacktriangleright
e^{\prime },r\blacktriangleright l) & \textrm{ and } & \,\,\,\,\,\,\,\,
(0,e)\blacktriangleright _{\bullet }(e^{\prime },l)=(ee^{\prime },\partial
_{1}(e)\blacktriangleright l-\{e^{\prime }\otimes e\}).
\end{align*}
We must  check conditions {\bf A1} and {\bf A2} of Definition \ref{action} (implicitly using lemma \ref{actionreduce}).
The second condition is clear, since all algebras are commutative. We prove the first: 
\begin{align*}
(r,0)&\blacktriangleright_{\bullet }\left[ (e^{\prime },l)(e_{2}^{\prime
},l_{2})\right]   =  (r,0)\blacktriangleright _{\bullet }(e^{\prime
}e_{2}^{\prime },e^{\prime }\blacktriangleright ^{\prime
}l_{2}+e_{2}^{\prime }\blacktriangleright' l+ll_{2}) \\ 
& =  (r\blacktriangleright (e^{\prime }e_{2}^{\prime
}),r\blacktriangleright (e^{\prime }\blacktriangleright ^{\prime
}l_{2}+e_{2}^{\prime }\blacktriangleright ^{\prime }l+ll_{2})) \\ 
& = \left( r\blacktriangleright (e^{\prime }e_{2}^{\prime
}),r\blacktriangleright (e^{\prime }\blacktriangleright ^{\prime
}l_{2})+r\blacktriangleright (e_{2}^{\prime }\blacktriangleright ^{\prime
}l)+r\blacktriangleright (ll_{2})\right)  \\ 
& =  (r\blacktriangleright (e^{\prime }e_{2}^{\prime
}),r\blacktriangleright \{e^{\prime }\otimes \partial
_{2}(l_{2})\}+r\blacktriangleright \{e_{2}^{\prime }\otimes \partial
_{2}(l)\}+r\blacktriangleright (ll_{2})).
\end{align*}
Also:
\begin{align*}
\left[ (r,0)\blacktriangleright _{\bullet }(e^{\prime },l)\right]&
(e_{2}^{\prime },l_{2})  = (r\blacktriangleright e^{\prime
},r\blacktriangleright l)(e_{2}^{\prime },l_{2}) \\ 
& = \left( (r\blacktriangleright e^{\prime })e_{2}^{\prime
},(r\blacktriangleright e^{\prime })\blacktriangleright ^{\prime
}l_{2}+e_{2}^{\prime }\blacktriangleright ^{\prime }(r\blacktriangleright
l)+(r\blacktriangleright l)l_{2}\right)  \\ 
& = (r\blacktriangleright (e^{\prime }e_{2}^{\prime
}),\{(r\blacktriangleright e^{\prime })\otimes \partial
_{2}(l_{2})\}+\{e_{2}^{\prime }\otimes ^{\prime }\partial
_{2}(r\blacktriangleright l)\}+r\blacktriangleright (ll_{2})) \\ 
& = (r\blacktriangleright (e^{\prime }e_{2}^{\prime
}),r\blacktriangleright \{e^{\prime }\otimes \partial
_{2}(l_{2})\}+\{e_{2}^{\prime }\otimes ^{\prime }r\blacktriangleright
\partial _{2}(l)\}+r\blacktriangleright (ll_{2})) \\ 
& = (r\blacktriangleright (e^{\prime }e_{2}^{\prime
}),r\blacktriangleright \{e^{\prime }\otimes \partial
_{2}(l_{2})\}+r\blacktriangleright \{e_{2}^{\prime }\otimes \partial
_{2}(l)\}+r\blacktriangleright (ll_{2})).
\end{align*}
This means that, for each $r \in R$, $e',e_2' \in E$ and $l,l_2 \in L$, we have:
\begin{equation*}
(r,0)\blacktriangleright _{\bullet }\left[ (e^{\prime },l)(e_{2}^{\prime
},l_{2})\right] =\left[ (r,0)\blacktriangleright (e^{\prime },l)\right]
(e_{2}^{\prime },l_{2}).%
\end{equation*}
Also, for each $r,r' \in R$ and $e' \in E$:
\begin{align*}
\left[ (r,0)(r^{\prime },0)\right] \blacktriangleright _{\bullet
}(e^{\prime },l) &= (rr^{\prime },0)\blacktriangleright _{\bullet
}(e^{\prime },l)  
= ((rr^{\prime })\blacktriangleright e^{\prime },(rr^{\prime
})\blacktriangleright l) \\ 
&= (r\blacktriangleright (r^{\prime }\blacktriangleright e^{\prime
}),r\blacktriangleright (r^{\prime }\blacktriangleright l))  = (r,0)\blacktriangleright _{\bullet }(r^{\prime }\blacktriangleright
e^{\prime },r^{\prime }\blacktriangleright l)\\& = (r,0)\blacktriangleright _{\bullet }\left[ (r^{\prime
},0)\blacktriangleright _{\bullet }(e^{\prime },l)\right].
\end{align*}
On the other hand: 
\begin{align*}
&(0,e)\blacktriangleright _{\bullet }(e^{\prime },0)=(ee^{\prime
},-\{e^{\prime }\otimes e\}),
&(0,e)\blacktriangleright _{\bullet }(0,l)=(0,\partial
_{1}(e)\blacktriangleright l).
\end{align*}
Therefore:
\begin{align*}
(0,e)\blacktriangleright _{\bullet }\left[ (e^{\prime },0)(e_{2}^{\prime
},0)\right]  &= (0,e)\blacktriangleright _{\bullet }(e^{\prime
}e_{2}^{\prime },0) = (e(e^{\prime }e_{2}^{\prime }),-\{e^{\prime }e_{2}^{\prime }\otimes
e\}), 
\end{align*}
\begin{align*}
&\left[ (0,e)\blacktriangleright _{\bullet }(e^{\prime },0)\right]
(e_{2}^{\prime },0) = (ee^{\prime },-\{e^{\prime }\otimes
e\})(e_{2}^{\prime },0) = ((ee^{\prime })e_{2}^{\prime },e_{2}^{\prime }\blacktriangleright
^{\prime }(-\{e^{\prime }\otimes e\})) \\ 
& \quad = (e(e^{\prime }e_{2}^{\prime }),\{e_{2}^{\prime }\otimes -\partial
_{2}\{e^{\prime },e\}\}) = (e(e^{\prime }e_{2}^{\prime }),\{e_{2}^{\prime }\otimes -ee^{\prime
}+\partial _{1}(e)\blacktriangleright e^{\prime }\}) \\ 
&\quad = (e(e^{\prime }e_{2}^{\prime }),\{e_{2}^{\prime }\otimes -ee^{\prime
}\})+(e(e^{\prime }e_{2}^{\prime }),\{e_{2}^{\prime }\otimes \partial
_{1}(e)\blacktriangleright e^{\prime }\}) \\ 
&\quad= (e(e^{\prime }e_{2}^{\prime }),-\{e^{\prime }e_{2}^{\prime }\otimes
e\})-(e(e^{\prime }e_{2}^{\prime }),\partial _{1}(e)\blacktriangleright
\{e_{2}^{\prime },e^{\prime }\})+(e(e^{\prime }e_{2}^{\prime }),\partial
_{1}(e)\blacktriangleright \{e_{2}^{\prime },e^{\prime }\}) \\ 
&\quad = (e(e^{\prime }e_{2}^{\prime }),-\{e^{\prime }e_{2}^{\prime }\otimes
e\}),
\end{align*}
which means that, for each $e,e',e_2' \in E$, we have:
\begin{equation*}
(0,e)\blacktriangleright _{\bullet }\left[ (e^{\prime },0)(e_{2}^{\prime },0)%
\right] =\left[ (0,e)\blacktriangleright _{\bullet }(e^{\prime },0)\right]
(e_{2}^{\prime },0).
\end{equation*}
In the same line:
\begin{align*}
(0,e)\blacktriangleright _{\bullet }[(0,l)(0,l_{2})] &= %
(0,e)\blacktriangleright _{\bullet }(0,ll_{2})=(0,\partial _{1}(e)\blacktriangleright ll_{2}),\\
\lbrack (e,0)\blacktriangleright _{\bullet }(0,l)](0,l_{2}) &= %
(0,\partial _{1}(e)\blacktriangleright l)(0,l_{2})= (0,(\partial _{1}(e)\blacktriangleright l)l_{2}) = (0,\partial _{1}(e)\blacktriangleright ll_{2}).
\end{align*}
This means that:
\begin{equation*}
(0,e)\blacktriangleright _{\bullet
}[(0,l)(0,l_{2})]=[(0,e)\blacktriangleright _{\bullet}(0,l)](0,l_{2}).
\end{equation*}%
Therefore we have:
\begin{equation*}
(0,e)\blacktriangleright _{\bullet }[(e^{\prime },l)(e_{2}^{\prime
},l_{2})]=[(0,e)\blacktriangleright _{\bullet }(e^{\prime
},l)](e_{2}^{\prime },l_{2}).
\end{equation*}
Similarly:
\begin{align*}
[ (0,e)&(0,e_{2})] \blacktriangleright _{\bullet }(e^{\prime },l)
= (0,ee_{2})\blacktriangleright _{\bullet }(e^{\prime },l) \\ 
&= ((ee_{2})e^{\prime },\partial _{1}(ee_{2})\blacktriangleright
l-\{e^{\prime }\otimes ee_{2}\}) = ((ee_{2})e^{\prime },\partial _{1}(ee_{2})\blacktriangleright
l-\{e^{\prime }\otimes e_{2}e\}) \\ 
&= ((ee_{2})e^{\prime },\partial _{1}(ee_{2})\blacktriangleright
l-\{e^{\prime }e_{2}\otimes e\}-\partial _{1}(e)\blacktriangleright
\{e^{\prime }\otimes e_{2}\}),
\end{align*}
\begin{align*}
&(0,e)\blacktriangleright _{\bullet }\left[ (0,e_{2})\blacktriangleright
_{\bullet }(e^{\prime },l)\right]  = (0,e)\blacktriangleright
_{\bullet }(e_{2}e^{\prime },\partial _{1}(e_{2})\blacktriangleright
l-\{e^{\prime }\otimes e_{2}\}) \\ 
&= (e(e_{2}e^{\prime }),\partial _{1}(e)\blacktriangleright \lbrack
\partial _{1}(e_{2})\blacktriangleright l-\{e^{\prime }\otimes
e_{2}\}]-\{e_{2}e^{\prime }\otimes e\}) \\ 
&= ((ee_{2})e^{\prime },\partial _{1}(e)\partial
_{1}(e_{2})\blacktriangleright l-\partial _{1}(e)\blacktriangleright
\{e^{\prime }\otimes e_{2}\}-\{e_{2}e^{\prime }\otimes e\}) \\ 
&= ((ee_{2})e^{\prime },\partial _{1}(ee_{2})\blacktriangleright
l-\partial _{1}(e)\blacktriangleright \{e^{\prime }\otimes
e_{2}\}-\{e_{2}e^{\prime }\otimes e\}).
\end{align*}
Thus: 
\begin{align*}
\left[ (0,e)(0,e_{2})\right] \blacktriangleright _{\bullet }(e^{\prime
},l)=(0,e)\blacktriangleright _{\bullet }\left[ (0,e_{2})\blacktriangleright
_{\bullet }(e^{\prime },l)\right],\forall e,e_2,e' \in E, \textrm{ and }  l \in L.
\end{align*}
Finally:
\begin{align*}
\left[ (r,0)(0,e_{2})\right] \blacktriangleright _{\bullet }(e^{\prime
},l)&=(0,r\blacktriangleright e_{2})\blacktriangleright _{\bullet }(e^{\prime
},l) \\
& =((r\blacktriangleright e_{2})e^{\prime },\partial
_{1}(r\blacktriangleright e_{2})\blacktriangleright l-\{e^{\prime }\otimes
r\blacktriangleright e_{2}\}) \\
& =((r\blacktriangleright e_{2})e^{\prime },\left( r\partial
_{1}(e_{2})\right) \blacktriangleright l-r\blacktriangleright \{e^{\prime
}\otimes e_{2}\}) \\
& =(r\blacktriangleright (e_{2}e^{\prime }),r\blacktriangleright (\partial
_{1}(e_{2})\blacktriangleright l)-r\blacktriangleright \{e^{\prime }\otimes
e_{2}\}) \\
& =(r,0)\blacktriangleright _{\bullet }(e_{2}e^{\prime },\partial
_{1}(e_{2})\blacktriangleright l-\{e^{\prime }\otimes e_{2}\}) \\
& =(r,0)\blacktriangleright _{\bullet }\left[ (0,e_{2})\blacktriangleright
_{\bullet }(e^{\prime },l)\right]. 
\end{align*}
\end{proof}
\noindent {Consider now the following
semidirect product (the algebra of 2-simplices):}%
\begin{equation}\label{2simp}
\A_2\doteq(R\ltimes _{\blacktriangleright}E)\ltimes _{\blacktriangleright
_{\bullet }}(E\ltimes _{\blacktriangleright^{\prime } }L).
\end{equation}
We  express the elements $(r,e,e^{\prime },l)\in
(R\ltimes _{\blacktriangleright}E)\ltimes _{\blacktriangleright
_{\bullet }}(E\ltimes _{\blacktriangleright^{\prime } }L)$ in the simplicial form:
\begin{equation}\label{2morsimp}
\xymatrix@R=15pt@C=15pt{\trinearlyn{l}{e^\prime}{e}{r}}
\end{equation}

\begin{remark}\label{b2}
We have three non-trivial algebra morphisms (boundaries):
\begin{equation*}
d_{0},d_1 ,d_2:\A_2=(R\ltimes _{\blacktriangleright }E)\ltimes
_{\blacktriangleright _{\bullet }}(E\ltimes _{\blacktriangleright^{\prime }
}L)\longrightarrow R\ltimes _{\blacktriangleright }E=\A_1,
\end{equation*}%
defined as:
\begin{align*}
d_{0}(r,e,e^{\prime },l) =(r,e)
&& d_{1}(r,e,e^{\prime },l) =(r,e+e^{\prime })
&& d_{2}(r,e,e^{\prime },l) =(r+\partial _{1}(e),e^{\prime }+\partial _{2}(l)).
\end{align*}
\begin{minipage}{0.8 \textwidth}
These maps are written in the simplicial notation below. Note that our convention for the numbering of the vertices in the triangle is:
\end{minipage}
\begin{minipage}{0.2\textwidth}
$\tiny{
\xymatrix@R=1pt@C=1pt{\twosimplex{2}{1}{0}}
}$
\end{minipage}

\bigskip

\begin{tabular}{ll}
$d_{0}\left(
\begin{tabular}{l}
$\xymatrix@R=15pt@C=15pt{\trinearly{l}{e^\prime}{e}{r}}$
\end{tabular}%
\right) $
\begin{tabular}{l}
$=$ $\left( \xymatrix{\pmoru{r}{e}}\right) ,$
\end{tabular}
\end{tabular}
\\
\\
\begin{tabular}{ll}\hskip0.4cm
$d_{1}\left(
\begin{tabular}{l}
$\xymatrix@R=15pt@C=15pt{\trinearly{l}{e^\prime}{e}{r}}$
\end{tabular} \right) $
\begin{tabular}{l}
$=$ $\left( \xymatrix{\pmoruu{r}{e+e^\prime}{r+\partial_{1}(e)+\partial_{1}(e^\prime)}}\right), $%
\end{tabular}
\\
\\
\hskip0.4cm $d_{2}\left(
\begin{tabular}{l}
$\xymatrix@R=15pt@C=15pt{\trinearly{l}{e^\prime}{e}{r}}$
\end{tabular}%
\right) $
\begin{tabular}{l}
$=$ $\left( \xymatrix{\pmoruu{r+\partial_{1}(e)}{e^\prime+\partial_{2}(l)}{r+\partial_{1}(e)+\partial_{1}(e^\prime)}}
\right) .$%
\end{tabular}
\end{tabular}
\end{remark}

\begin{remark}\label{2degeneracies}
Similarly, there exist two non-trivial $\k$-algebra morphisms (inclusions):
\begin{equation*}
s_{0},s_{1}\colon \A_1=R\ltimes _{\blacktriangleright }E \longrightarrow (R\ltimes _{\blacktriangleright }E)\ltimes
_{\blacktriangleright _{\bullet }}(E\ltimes _{\blacktriangleright^{\prime }}L)=\A_2,
\end{equation*}%
defined as:
\begin{align*}
&s_{0}(r,e)=(r,e,0,0), \ \ \
&&s_{1}(r,e)=(r,0,e,0),
\end{align*}%
and  simplicially visualized as:
\begin{align*}
{\xymatrix@R=15pt@C=15pt{\\ s_{0}\left(\s{r \overset{e}\longrightarrow (r+\partial_{1}(e))} \right) =}}
\xymatrix@R=15pt@C=15pt{\trinearlyy{0}{r}{r+\partial_{1}(e)}{r+%
\partial_{1}(e)}{e}{0}{e}} \xymatrix@R=18pt@C=15pt{\\  \textrm{ and } }
{ \xymatrix@R=15pt@C=15pt{\\ s_{1}\left(\s{r \overset{e}\longrightarrow (r+\partial_{1}(e))} \right) =}}%
\xymatrix@R=15pt@C=15pt{\trinearlyy{0}{r}{r}{r+%
\partial_{1}(e)}{0}{e}{e}}%
\end{align*}
\end{remark}

\begin{remark} {Let $\A$ be an algebra 2-crossed module. It is clear (from this simplicial notation) that the  morphisms we defined between the algebras of $0$-, $1$- and $2$-simplices in $\A$ satisfy the well known simplicial identities; \cite{PM2}. Therefore, we can form a 2-truncated simplicial commutative algebra, having at levels 0, 1 and 2, respectively, the algebra $R$ (the algebra of 0-simplices) and the algebras of 1- and 2-simplices.}
\end{remark}

\begin{remark}
This construction can be extended to a proof of the fact that we have an equivalence of categories between the category of 2-crossed modules of commutative algebras and the category of simplicial commutative algebras with Moore complex of length two. For a proof of this latter fact, {see} \cite{AP1,AP2,A1}; note that  our conventions are different. In the next subsection we inspect the algebra of 3-simplices. 
\end{remark}

\subsection{The algebra of 3-simplices in a 2-crossed module}\label{tet}
\noindent We continue to fix an algebra 2-crossed module $\A=(L,E,R,\partial _{1},\partial _{2},\{,\})$.
\subsubsection{Some auxiliary algebra actions}\label{aaa}
\begin{lemma}
There exists an action $\blacktriangleright _{\ast }$ of $E\ltimes L$ on $L$,
with:
\begin{equation*}
(e,l)\blacktriangleright _{\ast }l^{\prime }=e\blacktriangleright ^{\prime
}l^{\prime }+ll^{\prime }.
\end{equation*}
\end{lemma}

\begin{proof}
Easy calculations. 
\end{proof}

By using the action $\blacktriangleright _{\ast }$, we can construct the
semidirect product $
(E\ltimes _{\blacktriangleright ^{\prime }}L)\ltimes _{\blacktriangleright
_{\ast }}L$.
\begin{lemma}
There {exists} an action of $E$ on $\left( (E\ltimes _{\blacktriangleright
^{\prime }}L)\ltimes _{\blacktriangleright _{\ast }}L\right) $, with the form:
\begin{equation*}
e\blacktriangleright^{1}_{e} (e^{\prime },l,l^{\prime })=(ee^{\prime },\partial
_{1}(e)\blacktriangleright l-\{e^{\prime }\otimes e\},\partial
_{1}(e)\blacktriangleright l^{\prime })
\end{equation*}
\end{lemma}

\begin{proof}
We have:
\begin{align*}
& e\blacktriangleright^{1}_{e} \left[ (f,l,l^{\prime })(g,k,k^{\prime })\right]  \\
& = e\blacktriangleright^{1}_{e} (fg,f\blacktriangleright' k+g\blacktriangleright'
l+lk,f\blacktriangleright ^{\prime }k^{\prime }+lk^{\prime
}+g\blacktriangleright ^{\prime }l^{\prime }+kl^{\prime }+l'k') \\ 
& = \big (e(fg),\partial _{1}(e)\blacktriangleright (f\blacktriangleright'
k+g\blacktriangleright' l+lk)-\{fg\otimes e\},\\ &\phantom{------------} \partial
_{1}(e)\blacktriangleright (f\blacktriangleright ^{\prime }k^{\prime
}+lk^{\prime }+g\blacktriangleright ^{\prime }l^{\prime }+kl^{\prime }+l'k'\big) \\ 
& = \big (e(fg),\partial _{1}(e)\blacktriangleright (f\blacktriangleright'
k)+\partial _{1}(e)\blacktriangleright (g\blacktriangleright' l)+\partial
_{1}(e)\blacktriangleright (lk)-\{fg\otimes e\}, \\  &\phantom{--} (\partial
_{1}(e) f)\blacktriangleright ^{\prime }k^{\prime
}+\partial _{1}(e)\blacktriangleright (lk^{\prime })+(\partial
_{1}(e)g)\blacktriangleright ^{\prime }l^{\prime
}+\partial _{1}(e)\blacktriangleright (kl^{\prime })+\partial _{1}(e)\blacktriangleright (k'l')\big) \\ 
& = \big((ef)g,(ef)\blacktriangleright' k+g\blacktriangleright' (\partial
_{1}(e)\blacktriangleright l-\{f\otimes e\})+(\partial
_{1}(e)\blacktriangleright l-\{f\otimes e\})k,\\ & \quad (ef)\blacktriangleright'
k^{\prime }+(\partial _{1}(e)\blacktriangleright l-\{f\otimes e\})k^{\prime
}+g\blacktriangleright' (\partial _{1}(e)\blacktriangleright l^{\prime
})\\ & \phantom{------------}  +k(\partial _{1}(e)\blacktriangleright l^{\prime }+(\partial
_{1}(e)\blacktriangleright l')k' \big)\\ 
& = \big(ef,\partial _{1}(e)\blacktriangleright l-\{f\otimes e\},\partial
_{1}(e)\blacktriangleright l^{\prime }\big)\, \big (g,k,k^{\prime }\big) = \left[ e\blacktriangleright^{1}_{e} (f,l,l^{\prime })\right] (g,k,k^{\prime
}\big).
\end{align*}
We used the facts:
\begin{align*}
\partial _{1}(e)\blacktriangleright (f\blacktriangleright ^{\prime }k) &= \partial _{1}(e)\blacktriangleright \{f\otimes \partial _{2}(k)\} = \{\partial _{1}(e)\blacktriangleright f\otimes \partial _{2}(k)\} \\
& =\{fe-\partial _{2}\{f\otimes e\}\otimes \partial _{2}(k)\} 
=\{fe\otimes \partial _{2}(k)\}-\{\partial _{2}\{f\otimes e\}\otimes
\partial _{2}(k)\} \\
& = (fe)\blacktriangleright ^{\prime }k-\{f\otimes e\}k,
\end{align*}
and:
\begin{align*}
\{fg\otimes e\} & = \{gf\otimes e\} = \{g\otimes fe\}-\partial _{1}(e)\blacktriangleright \{g\otimes f\}= \{g\otimes fe\}-\{g\otimes \partial _{1}(e)\blacktriangleright f\} \\
& = \{g\otimes fe-\partial _{1}(e)\blacktriangleright f\} = \{g\otimes \partial _{2}\{f\otimes e\}\} = g\blacktriangleright ^{\prime }\{f\otimes e\}.
\end{align*}
Also:
\begin{align*}
(e_{1}e_{2}) & \blacktriangleright^{1}_{e} (e^{\prime },l,l^{\prime }) = %
((e_{1}e_{2})e^{\prime },\partial _{1}(e_{1}e_{2})\blacktriangleright
l-\{e^{\prime }\otimes (e_{1}e_{2})\},\partial
_{1}(e_{1}e_{2})\blacktriangleright l^{\prime }) \\ 
& = \left( (e_{1}(e_{2}e^{\prime }),\partial
_{1}(e_{1})\blacktriangleright \left( \partial
_{1}(e_{2})\blacktriangleright l\right) -\{e^{\prime }\otimes
(e_{1}e_{2})\},\partial _{1}(e_{1})\blacktriangleright \left( \partial
_{1}(e_{2})\blacktriangleright l^{\prime }\right) \right)  \\
& = ( e_{1}(e_{2}e^{\prime }),\partial
_{1}(e_{1})\blacktriangleright (\partial _{1}(e_{2})\blacktriangleright
l)-\partial _{1}(e_{1})\blacktriangleright \{e^{\prime }\otimes
e_{2}\}) \\
& \quad \quad \quad - \{e_{2}e^{\prime }\otimes e_{1}\},\partial
_{1}(e_{1})\blacktriangleright (\partial _{1}(e_{2})\blacktriangleright
l^{\prime })) \quad \quad (\because \textbf{2XM3}) \\ 
& = (e_{1}(e_{2}e^{\prime }),\partial
_{1}(e_{1})\blacktriangleright (\partial _{1}(e_{2})\blacktriangleright
l-\{e^{\prime }\otimes e_{2}\})-\{e_{2}e^{\prime }\otimes e_{1}\}, \\
& \quad \quad \quad \quad \partial_{1}(e_{1})\blacktriangleright (\partial _{1}(e_{2})\blacktriangleright
l^{\prime }))  \\ 
& = e_{1}\blacktriangleright^{1}_{e} \left( e_{2}e^{\prime },\partial
_{1}(e_{2})\blacktriangleright l-\{e^{\prime }\otimes e_{2}\},\partial
_{1}(e_{2})\blacktriangleright l^{\prime }\right)  \\ 
& = e_{1}\blacktriangleright^{1}_{e} (e_{2}\blacktriangleright^{1}_{e} (e^{\prime
},l,l^{\prime })).
\end{align*}
\end{proof}

Simple calculations also prove:
\begin{lemma}
There exists an action of $R$ on $\left( (E\ltimes _{\blacktriangleright
^{\prime }}L)\ltimes _{\blacktriangleright _{\ast }}L\right) $, with the form:
\begin{align*}
r\blacktriangleright^{1}_{r} (e^{\prime },l,l^{\prime })=(r\blacktriangleright
e^{\prime },r\blacktriangleright l,r\blacktriangleright l^{\prime }).
\end{align*}
\end{lemma}
 From the previous two lemmas, it follows that:
\begin{lemma}\label{A31}
There exists an action $\blacktriangleright ^{1}$ of $(R\ltimes
_{\blacktriangleright }E)$ on $\left( (E\ltimes _{\blacktriangleright
^{\prime }}L)\ltimes _{\blacktriangleright _{\ast }}L\right) $, with:
\begin{equation*}
(r,e)\blacktriangleright ^{1}(e^{\prime },l,l^{\prime
})=(r\blacktriangleright e^{\prime }+ee^{\prime },r\blacktriangleright
l+\partial _{1}(e)\blacktriangleright l-\{e^{\prime }\otimes
e\},r\blacktriangleright l^{\prime }+\partial _{1}(e)\blacktriangleright
l^{\prime }).
\end{equation*}
\end{lemma}\label{lemmaref}
Yet, one more {(seemingly unrelated)} action is  needed:
\begin{lemma}
There exists yet another action of $E$ on $\left( (E\ltimes _{\blacktriangleright
^{\prime }}L)\ltimes _{\blacktriangleright _{\ast }}L\right) $, with:
\begin{equation*}
e\blacktriangleright^{2}_{e} (e',l,l^{\prime })=(ee',e\blacktriangleright' l,\partial
_{1}(e)\blacktriangleright l^{\prime }-\{\partial _{2}(l)+e'\otimes e\}).
\end{equation*}
\end{lemma}

\begin{proof}
Firstly, $E$ acts on $\left( (\{0\}\ltimes _{\blacktriangleright
^{\prime }}\{0\})\ltimes _{\blacktriangleright _{\ast }}L\right)$, in the form:
\begin{align*}
e\blacktriangleright^{2}_{e} \left[ (0,0,l^{\prime })(0,0,k^{\prime })\right] & =
e\blacktriangleright^{2}_{e} (0,0,l^{\prime }k^{\prime })
= (0,0,\partial _{1}(e)\blacktriangleright (l^{\prime }k^{\prime }))
= (0,0,(\partial _{1}(e)\blacktriangleright l^{\prime })k^{\prime })
\\
& = (0,0,\partial _{1}(e)\blacktriangleright l^{\prime })(0,0,k^{\prime
})
= \left[ e\blacktriangleright^{2}_{e} (0,0,l^{\prime })\right] (0,0,k^{\prime
}).
\end{align*}%
Then, we need to show that $E$ acts on $\left( (E\ltimes _{\blacktriangleright
^{\prime }}L)\ltimes _{\blacktriangleright _{\ast }}\{0\}\right)$.
We consider two cases:
\begin{align*}
e\blacktriangleright^{2}_{e} \left[ (0,l,0)(0,k,0)\right]  & = %
e\blacktriangleright^{2}_{e} (0,lk,0)
= (0,e\blacktriangleright' (lk),-\{\partial _{2}(lk)\otimes e\}) \\
& = (0,e\blacktriangleright' l,-\{\partial _{2}(l)\otimes e\})(0,k,0)
= \left[ e\blacktriangleright^{2}_{e} (0,l,0)\right] (0,k,0),
\end{align*}
and also:
\begin{align*}
e\blacktriangleright^{2}_{e} \left[ (f,0,0)(g,0,0)\right]& =
e\blacktriangleright^{2}_{e} (fg,0,0)
= (e(fg),0,-\{fg\otimes e\})\\  &= ((ef)g,0,-g\blacktriangleright'
\{f\otimes e\}) = ((ef)g,0,g\blacktriangleright ^{\prime }(-\{f\otimes e\}))\\
&= (ef,0,)(g,0,0)
= \left[ e\blacktriangleright^{2}_{e} (f,0,0)\right] (g,0,0).
\end{align*}
We have used the fact:
\begin{align*}
& g\blacktriangleright' \{f\otimes e\}=\{g\otimes \partial
_{2}\{f\otimes e\}\}=\{g\otimes fe-\partial _{1}(e)\blacktriangleright f\} \\
& \quad =\{g\otimes fe\}-\{g\otimes \partial _{1}(e)\blacktriangleright
f\}=\{gf\otimes e\}+\partial _{1}(e)\blacktriangleright \{g\otimes
f\}-\{g\otimes \partial _{1}(e)\blacktriangleright f\} \\
& \quad =\{gf\otimes e\}+\partial _{1}(e)\blacktriangleright \{g\otimes
f\}-\partial _{1}(e)\blacktriangleright \{g\otimes f\}=\{gf\otimes
e\}=\{fg\otimes e\}.
\end{align*}
Also (where we use the axiom {\bf 2XM3} of the definition of 2-crossed modules):%
\begin{align*}
& (e_{1}e_{2})\blacktriangleright^{2}_{e} (e^{\prime },l,l^{\prime }) =
((e_{1}e_{2})e^{\prime },(e_{1}e_{2})\blacktriangleright' l,\partial
_{1}(e_{1}e_{2})\blacktriangleright l^{\prime }-\{\partial _{2}(l)+e^{\prime
}\otimes (e_{1}e_{2})\}) \\
& = \big(e_{1}(e_{2}e^{\prime }),e_{1}\blacktriangleright'
(e_{2}\blacktriangleright' l),\\ & \phantom{---}\partial _{1}(e_{1})\blacktriangleright
(\partial _{1}(e_{2})\blacktriangleright l^{\prime }- \{\partial
_{2}(l)+e^{\prime }\otimes e_{2}\})-\{\partial _{2}(e_{2}\blacktriangleright
l)+e_{2}e^{\prime }\otimes e_{1}\}\big) \\
& = e_{1}\blacktriangleright^{2}_{e} (e_{2}e^{\prime },e_{2}\blacktriangleright'
l,\partial _{1}(e_{2})\blacktriangleright l^{\prime }-\{\partial
_{2}(l)+e^{\prime }\otimes e_{2}\}) = e_{1}\blacktriangleright^{2}_{e} (e_{2}\blacktriangleright^{2}_{e} (e^{\prime
},l,l^{\prime })).
\end{align*}
\end{proof}

\medskip

\noindent {The following lemma is immediate. The subsequent follows from it and Lemma \ref{lemmaref}.} \begin{lemma}
There exists an action of $L$ on $\left( (E\ltimes _{\blacktriangleright
^{\prime }}L)\ltimes _{\blacktriangleright _{\ast }}L\right) $, with:
\begin{align*}
k\blacktriangleright^{2}_{l} (e^{\prime },l,l^{\prime })=(0,e^{\prime
}\blacktriangleright' k+kl,-\{\partial _{2}(l)+e^{\prime }\otimes \partial
_{2}(k)\}).
\end{align*}
\end{lemma}

\begin{lemma}\label{A32}
There exists an action $\blacktriangleright ^{2}$ of $(E\ltimes
_{\blacktriangleright ^{\prime }}L)$ on $\left( (E\ltimes
_{\blacktriangleright ^{\prime }}L)\ltimes _{\blacktriangleright _{\ast
}}L\right) $ with:
\begin{equation*}
(e,l^{\prime \prime })\blacktriangleright ^{2}(e^{\prime },l,l^{\prime
})=(ee^{\prime },e\blacktriangleright' l+e^{\prime }\blacktriangleright'
l^{\prime \prime }+l^{\prime \prime }l,\partial _{1}(e)\blacktriangleright
l^{\prime }-\{\partial _{2}(l)+e^{\prime }\otimes \partial _{2}(l^{\prime
\prime })+e\}
\end{equation*}
\end{lemma}
And, finally {(from Lemma \ref{A31} and Lemma \ref{A32}):}
\begin{lemma}
There exists an action $\blacktriangleright _{\dagger }$ of $\left(
(R\ltimes _{\blacktriangleright }E)\ltimes _{\blacktriangleright _{\bullet
}}(E\ltimes _{\blacktriangleright ^{\prime }}L)\right) $ on \linebreak $\left(
(E\ltimes _{\blacktriangleright ^{\prime }}L)\ltimes _{\blacktriangleright
_{\ast }}L\right) $, with the form:%
\begin{align*}
(0,0,e,l^{\prime \prime })\blacktriangleright _{\dagger
}(e^{\prime },l,l^{\prime })&=(e,l^{\prime \prime
})\blacktriangleright ^{2}(e^{\prime },l,l^{\prime }),\\
(r,e,0,0)\blacktriangleright _{\dagger }(e^{\prime },l,l^{\prime
})&=(r,e)\blacktriangleright ^{1}(e^{\prime },l,l^{\prime }).
\end{align*}
\end{lemma}

\subsubsection{Definition of the algebra of 3-simplices}
We now define the algebra of 3-simplices in $\A$ as being the semidirect  product:%
\begin{align}
\A_3\doteq\left( (R\ltimes _{\blacktriangleright }E)\ltimes _{\blacktriangleright
_{\bullet }}(E\ltimes _{\blacktriangleright ^{\prime }}L)\right) \ltimes
_{\blacktriangleright _{\dagger }}\left( (E\ltimes _{\blacktriangleright
^{\prime }}L)\ltimes _{\blacktriangleright _{\ast }}L\right).
\end{align}
We  express each element $(r,e,e^{\prime },l,e^{\prime \prime },l^{\prime },l^{\prime \prime })\in \A_3$
in the following simplicial form:
\begin{equation}\label{tet1}
\xymatrix@R=25pt@C=38pt{\tetrn{r}{e}{e^{\prime}}{l}{e^{\prime\prime}}{l^{\prime}}{l^{\prime\prime}}}
\end{equation}

\begin{remark}\label{tetfaces}
We have four non-trivial algebra morphisms, $d_{0,1,2,3}\colon \A_3 \to \A_2$,
being:\begin{align*}
d_{0}(r,e,e^{\prime },l,e^{\prime \prime },l^{\prime },l^{\prime \prime })
&=(r,e,e^{\prime },l), \\
d_{1}(r,e,e^{\prime },l,e^{\prime \prime },l^{\prime },l^{\prime \prime })
&=(r,e,e^{\prime }+e^{\prime \prime },l+l^{\prime }), \\
d_{2}(r,e,e^{\prime },l,e^{\prime \prime },l^{\prime },l^{\prime \prime })
&=(r,e+e^{\prime },e^{\prime \prime },l^{\prime }+l^{\prime \prime }), \\
d_{3}(r,e,e^{\prime },l,e^{\prime \prime },l^{\prime },l^{\prime \prime })
&=(r+\partial _{1}(e),e^{\prime }+\partial _{2}(l),e^{\prime \prime
}+\partial _{2}(l^{\prime }),l^{\prime \prime }).
\end{align*}

\begin{minipage}{0.8\textwidth}{
These should of course be {interpreted} as being the faces of the tetrahedron \eqref{tet1} (see above), where we  {enumerate} the vertices of the tetrahedron as shown on the right:}
\end{minipage} 
\begin{minipage}{0.2\textwidth}
$
\xymatrix@R=5pt@C=5pt{\threesimplex{3}{2}{0}{1}}\xymatrix@R=5pt@C=10pt{\\\quad \quad }
$
\end{minipage}

\hskip 1cm
\begin{tabular}{lll}
$d_{0}\left(
\begin{tabular}{l}
\xymatrix@R=15pt@C=20pt{\threesimplex{\bullet}{\bullet}{\bullet}{\bullet}}
\end{tabular}%
\right)=$ &
\begin{tabular}{l}
\xymatrix{\triangl{r}{e}{e^{\prime}}{l}}
\end{tabular} \\
\end{tabular}

\hskip 1cm
\begin{tabular}{lll}
$d_{1}\left(
\begin{tabular}{l}
\xymatrix@R=15pt@C=20pt{\threesimplex{\bullet}{\bullet}{\bullet}{\bullet}}
\end{tabular}%
\right)=$ &
\begin{tabular}{l}
\xymatrix{\triangly{r}{r+\partial_{1}(e)}{r+\partial_{1}(e)+\partial_{1}(e^{\prime})+\partial_{1}(e^{\prime\prime})}{e}{e^{\prime}+e^{\prime\prime}+\partial_{2}(l)+\partial_{2}(l^{\prime})}{e+e^{\prime}+e^{\prime\prime}}{l+l^{\prime}}}
\end{tabular}
\end{tabular}

\hskip 1cm
\begin{tabular}{lll}
$d_{2}\left(
\begin{tabular}{l}
\xymatrix@R=15pt@C=20pt{\threesimplex{\bullet}{\bullet}{\bullet}{\bullet}}
\end{tabular}%
\right)=$ &
\begin{tabular}{l}
\xymatrix{\triangly{r}{r+\partial_{1}(e)+\partial_{1}(e^{\prime})}{r+\partial_{1}(e)+\partial_{1}(e^{\prime})+\partial_{1}(e^{\prime\prime})}{e+e^{\prime}}{e^{\prime\prime}+\partial_{2}(l^{\prime})+\partial_{2}(l^{\prime\prime})}{e+e^{\prime}+e^{\prime\prime}}{l^{\prime}+l^{\prime\prime}}}
\end{tabular}
\end{tabular}

\hskip 1cm
\begin{tabular}{lll}
$d_{3}\left(
\begin{tabular}{l}
\xymatrix@R=15pt@C=20pt{\threesimplex{\bullet}{\bullet}{\bullet}{\bullet}}
\end{tabular}%
\right)=$ &
\begin{tabular}{l}\hskip-0.5cm
\xymatrix{\triangly{r+\partial_{1}(e)}{r+\partial_{1}(e)+\partial_{1}(e^{\prime})}{r+\partial_{1}(e)+\partial_{1}(e^{\prime})+\partial_{1}(e^{\prime\prime})}{e^{\prime}+\partial_{2}(l)}{e^{\prime\prime}+\partial_{2}(l^{\prime})+\partial_{2}(l^{\prime\prime})}{e^{\prime}+e^{\prime\prime}+\partial_{2}(l)+\partial_{2}(l^{\prime})\quad \quad }{l^{\prime\prime}}}
\end{tabular}
\end{tabular}
\end{remark}

\begin{remark}
There are three algebra morphisms, 
$
s_{0,1,2}\colon \A_2 \to \A_3$, called degeneracies, or inclusions, 
from the algebra of 2-simplices to the algebra of 3-simplices, being:
\begin{align*}
s_{0}(r,e,e^{\prime },l) &=(r,e,e^{\prime },l,0,0,0), \\
s_{1}(r,e,e^{\prime },l) &=(r,e,0,0,e^{\prime },l,0), \\
s_{2}(r,e,e^{\prime },l) &=(r,0,e,0,e^{\prime },0,l).
\end{align*}%
These are simplicially visualized (in a very clear way) as:

\begin{tabular}{lll}
\hskip-0cm
$s_{0}\left(
\begin{tabular}{l}
$\xymatrix@R=15pt@C=15pt{\trinearly{l}{e^\prime}{e}{r}}$
\end{tabular}%
\right) =$ &
\hskip0cm
\begin{tabular}{l}
\xymatrix@R=20pt@C=22pt{
& & \s{r+\partial_{1}(e)+\partial_{1}(e^{\prime})} & & \\
& & & & \\
& & \s{r+\partial_{1}(e)+\partial_{1}(e^{\prime})} \ar[uu]^{\s{0}} & & \\
\s{r} \ar[rrrr]_{\s{e}}^{\ovalbox{$\s{l}$}} \ar[rru]_{\s{e+e^{\prime}}} \ar[rruuu]^{\s{e+e^{\prime}}}_{\ovalbox{$\s{0}$}} & & & & \s{r+\partial_{1}(e)} \ar[llu]^{\s{e^{\prime}+\partial_{2}(l)}} \ar[lluuu]_{\s{e^{\prime}}}^{\ovalbox{$\s{0}$}}
}
\end{tabular}
\end{tabular}

\begin{tabular}{lll}\hskip-0cm

$s_{1}\left(
\begin{tabular}{l}
$\xymatrix@R=15pt@C=15pt{\trinearly{l}{e^\prime}{e}{r}}$
\end{tabular}%
\right) =$ &
\hskip0cm
\begin{tabular}{l}
\xymatrix@R20pt@C=22pt{
& & \s{r+\partial_{1}(e)+\partial_{1}(e^{\prime})} & & \\
& & & & \\
& & \s{r+\partial_{1}(e)} \ar[uu]|{\s{e^{\prime}+\partial_{2}(l)}} & & \\
\s{r} \ar[rrrr]_{\s{e}}^{\ovalbox{$\s{0}$}} \ar[rru]_{\s{e}} \ar[rruuu]^{\s{e+e^{\prime}}}_{\ovalbox{$\s{l}$}} & & & & \s{r+\partial_{1}(e)} \ar[llu]^{\s{0}} \ar[lluuu]_{\s{e^{\prime}+\partial_{2}(l)}}^{\ovalbox{$\s{0}$}}
}
\end{tabular}
\end{tabular}

\begin{tabular}{ll}\hskip-0cm

$s_{2}\left(
\begin{tabular}{l}
$\xymatrix@R=15pt@C=15pt{\trinearly{l}{e^\prime}{e}{r}}$
\end{tabular}%
\right) =$ &
\hskip0cm
\begin{tabular}{l}
\xymatrix@R20pt@C=25pt{
& & \s{r+\partial_{1}(e)+\partial_{1}(e^{\prime})} & & \\
& & & & \\
& & \s{r+\partial_{1}(e)} \ar[uu]|{\s{e^{\prime}+\partial_{2}(l)}} & & \\
\s{r} \ar[rrrr]_{\s{0}}^{\ovalbox{$\s{0}$}} \ar[rru]_{\s{e}} \ar[rruuu]^{\s{e+e^{\prime}}}_{\ovalbox{$\s{l}$}} & & & & \s{r} \ar[llu]^{\s{e}} \ar[lluuu]_{\s{e+e^{\prime}}}^{\ovalbox{$\s{l}$}}
}
\end{tabular}
\end{tabular}
\end{remark}
%
%

\section{Pointed homotopy of crossed module maps}

\noindent In this section,  we fix two algebra crossed modules $\mathcal{A}%
=(E,R,\partial )$ and $\mathcal{A}^{\prime }=(E^{\prime },R^{\prime
},\partial ^{\prime })$.
\subsection{Derivations and homotopy between crossed module maps}
\begin{definition}
Let $f_{0}\colon R\to R^{\prime }$ be an algebra homomorphism. An $%
f_{0}$-derivation $s\colon R\to E^{\prime }$ is a $\k$-linear map satisfying, for all $r,r^{\prime }\in R$:
\begin{align}
s(rr^{\prime })=f_{0}(r)\blacktriangleright s(r^{\prime })+f_{0}(r^{\prime
})\blacktriangleright s(r)+s(r)s(r^{\prime }).
\end{align}
\end{definition}

\begin{theorem}[Pointed homotopy of crossed module maps]
Let $f$ be a crossed module morphism $\mathcal{A\to A}^{\prime }$. In the condition of the
previous definition, if $s$ is an $f_{0}$-derivation, and if we define $%
g=(g_{1},g_{0})$ as  (where $e\in E$ and $r\in R$):
\begin{align}\label{hom}
&g_{0}(r)=f_{0}(r)+(\partial ^{\prime } \circ s)(r), 
&&g_{1}(e)=f_{1}(e)+(s \circ \partial )(e),
\end{align}
 then $g$ is also a crossed module morphism $\mathcal{A\to A}^{\prime }$. In such a case we write:
$f\ra{(f_{0},s)}g,$
and say that $(f_{0},s)$ is a homotopy (or derivation)  connecting $f$ to $g$. 
\end{theorem}


\begin{proof}
We first show that $g_{0}$ and $g_{1}$ are algebra
morphisms. That $g_{0}(r+r^{\prime })=g_{0}(r)+g_{0}(r^{\prime })$ and also $g_{0}(kr)=kg_{0}(r)$, follows from $\k$-linearity, and similarly for $g_1$.
 {It is also clear that $g_0\circ \d = \d' \circ g_1$. In addition:}
%
%
\begin{align*}
g_{0}(rr^{\prime }) & = f_{0}(rr^{\prime })+(\partial ^{\prime }\circ
s)(rr^{\prime }) = f_{0}(r)f_{0}(r^{\prime })+\partial ^{\prime }(s(rr^{\prime })) \\
& = f_{0}(r)f_{0}(r^{\prime })+\partial ^{\prime
}(f_{0}(r)\blacktriangleright s(r^{\prime })+f_{0}(r^{\prime
})\blacktriangleright s(r)+s(r)s(r^{\prime })) \\
& = f_{0}(r)f_{0}(r^{\prime })+\partial ^{\prime
}(f_{0}(r)\blacktriangleright s(r^{\prime }))+\partial ^{\prime
}(f_{0}(r^{\prime })\blacktriangleright s(r))+\partial ^{\prime
}(s(r)s(r^{\prime })) \\
& = f_{0}(r)f_{0}(r^{\prime })+f_{0}(r)\partial ^{\prime }(s(r^{\prime
}))+f_{0}(r^{\prime })\partial ^{\prime }(s(r))+\partial ^{\prime
}(s(r))\partial ^{\prime }(s(r^{\prime })) \\
& = f_{0}(r)f_{0}(r^{\prime })+f_{0}(r)(\partial ^{\prime }\circ
s)(r^{\prime })+f_{0}(r^{\prime })(\partial ^{\prime } \circ s)(r)+(\partial
^{\prime }\circ s)(r)(\partial ^{\prime }\circ s)(r^{\prime }) \\
& = [f_{0}(r)+(\partial ^{\prime }\circ s)(r)][f_{0}(r^{\prime
})+(\partial ^{\prime }\circ s)(r^{\prime })] = g_{0}(r)g_{0}(r^{\prime }),
\end{align*}%
for all $r,r^{\prime }\in R$, thus $g_{0}$ is an algebra morphism.
Similarly,  $g_{1}$ is an algebra morphism.
Finally, $g\colon \A_1 \to \A_2$ preserves the action of $R$ on $E$, since for each $r\in R$ and $e\in E$: %
\begin{align*}
g_{1}(r\blacktriangleright e) & = f_{1}(r\blacktriangleright e)+s(\partial (r\blacktriangleright
e)) = f_{0}(r)\blacktriangleright f_{1}(e)+s(r\partial (e)) \\
& = f_{0}(r)\blacktriangleright f_{1}(e)+f_{0}(r)\blacktriangleright
s(\partial (e))+f_{0}(\partial (e))\blacktriangleright
s(r)+s(r)s(\partial (e)) \\
& = f_{0}(r)\blacktriangleright f_{1}(e)+f_{0}(r)\blacktriangleright
s(\partial (e))+\partial ^{\prime }(f_{1}(e))\blacktriangleright
s(r)+s(r)s(\partial (e)) \\
& = f_{0}(r)\blacktriangleright f_{1}(e)+f_{0}(r)\blacktriangleright
s(\partial (e))+f_{1}(e)s(r)+s(r)s(\partial (e)) \\
& = f_{0}(r)\blacktriangleright f_{1}(e)+f_{0}(r)\blacktriangleright
s(\partial (e))+s(r)f_{1}(e)+s(r)s(\partial (e)) \\
& = f_{0}(r)\blacktriangleright f_{1}(e)+f_{0}(r)\blacktriangleright
s(\partial (e))+\partial ^{\prime }(s(r)\blacktriangleright
f_{1}(e))+\partial ^{\prime }(s(r)\blacktriangleright s(\partial (e))) \\
& = f_{0}(r)\blacktriangleright \lbrack f_{1}(e)+(s\circ \partial
)(e)]+\partial ^{\prime }(s(r))\blacktriangleright \lbrack f_{1}(e)+(s\circ
\partial )(e)] \\
& = [f_{0}(r)+(\partial ^{\prime }\circ s)(r)]\blacktriangleright \lbrack
f_{1}(e)+(s\circ \partial )(e)] = g_{0}(r)\blacktriangleright g_{1}(e).
\end{align*}
{In the last two calculations, we used the  second {Peiffer-Whitehead} law {\bf XM2}; Definition \ref{pcm}. Therefore they  are not true, in general, in the pre-crossed module case.}
\end{proof}

\begin{remark}
 {The prefix ``pointed'' appears in analogy with the group case, where pointed homotopies of crossed modules correspond to pointed homotopies of based spaces. We can consider a more general notion of homotopy of crossed modules of algebras, which parallels the non-pointed homotopy of crossed modules of groups, which model the case when the base-point of a space does not stay fixed throughout a homotopy. In such non-pointed context, we should add to a  homotopy  $(f_0,s)$ an element $b\in R'$, which should be a projector, namely $b^2=b$. We modify \eqref{hom}, as: }
 \begin{align}
&g_{0}(r)=b \big(f_{0}(r)+(\partial ^{\prime } \circ s)(r)\big), 
&&g_{1}(e)=b \t \big (f_{1}(e)+(s \circ \partial )(e)\big).
\end{align}
{We expect to fully address non-pointed homotopy of crossed modules, and 2-crossed modules, of commutative algebras in a future publication.} 
\end{remark}

\subsection{A groupoid of crossed module maps and their homotopies}\label{homotopies}

\begin{lemma}
Let $f=(f_1,f_0)$ be a crossed module morphism $\mathcal{A\to A}%
^{\prime }$. Then the null function $
0_{s}: r \in R  \longmapsto 0_{E^{\prime }} \in   E^{\prime } 
$
defines an $f_{0}$-derivation connecting $f$ to $f$.
\end{lemma}

\begin{lemma}
Let $f=(f_1,f_0)$ and $g=(g_1,g_0)$ be crossed module morphisms $\mathcal{%
A\to A}^{\prime }$ and $s$ be an $f_{0}$-derivation connecting $%
f $ to $g$. Then, the linear map $\bar{s}=-s\colon R \to E'$, with 
$\bar{s}(r)=-s(r)$, where $r \in R$,
is a $g_{0}$-derivation connecting $g$ to $f$.
\end{lemma}

\begin{proof}
Since $s$ is an $f_{0}$-derivation connecting $f$ to $g$, we have:
\begin{align*}
&f_{0}(r)=g_{0}(r)+(\partial ^{\prime }\circ \bar{s})(r), &\textrm{ and } &&
f_{1}(r)=g_{1}(r)+(\bar{s}\circ \partial )(r).
\end{align*}
Moreover $\bar{s}$ is an $g_{0}$-derivation, since:
\begin{align*}
\bar{s}(rr^{\prime }) &= -(s(rr^{\prime })) = -(f_{0}(r)\blacktriangleright s(r^{\prime }))-(f_{0}(r^{\prime
})\blacktriangleright s(r))-(s(r)s(r^{\prime })) \\
&= -(f_{0}(r)\blacktriangleright s(r^{\prime }))-(f_{0}(r^{\prime
})\blacktriangleright s(r))-s(r)s(r^{\prime })+s(r)s(r^{\prime
})-s(r)s(r^{\prime }) \\
&= -(f_{0}(r)\blacktriangleright s(r^{\prime }))-((\partial ^{\prime
} \circ s)(r)\blacktriangleright s(r^{\prime }))-(f_{0}(r^{\prime
})\blacktriangleright s(r)) \\
&   \quad \quad \quad  -((\partial ^{\prime } \circ s)(r^{\prime
})\blacktriangleright s(r))+s(r)s(r^{\prime }) \quad  \bf{(\because 2^{nd} \textbf{ {Peiffer-Whitehead} law})} \\
&= -(f_{0}(r)+(\partial ^{\prime } \circ s)(r))\blacktriangleright s(r^{\prime
})-(f_{0}(r)+(\partial ^{\prime } \circ s)(r))\blacktriangleright
s(r)+s(r)s(r^{\prime }) \\
&= (f_{0}(r)+(\partial ^{\prime } \circ s)(r))\blacktriangleright -s(r^{\prime
})+(f_{0}(r)+(\partial ^{\prime } \circ s)(r))\blacktriangleright
-s(r)+s(r)s(r^{\prime }) \\
&= g_{0}(r)\blacktriangleright -s(r^{\prime })+g_{0}(r^{\prime
})\blacktriangleright -s(r)+s(r)s(r^{\prime }) \\
&= g_{0}(r)\blacktriangleright \bar{s}(r^{\prime })+g_{0}(r^{\prime
})\blacktriangleright \bar{s}(r)+\bar{s}(r)\bar{s}(r^{\prime }).
\end{align*}
%
Note that we explicitly used the  second {Peiffer-Whitehead} law {\bf XM2}; Definition \ref{pcm}.
\end{proof}

\begin{lemma}[Concatenation of derivations] Let $f,g$ and $h$ be crossed module
morphisms $\mathcal{A\to A}^{\prime }$, $s$ be an $f_{0}$%
-derivation connecting $f$ to $g$, and $s^{\prime }$ be a $g_{0}$%
-derivation connecting $g$ to $h$. Then the linear map $(s+s')\colon R \to E'$, such that $(s+s^{\prime })(r)=s(r)+s^{\prime }(r)$, defines an $f_{0}$-derivation (therefore a homotopy) connecting $f$ to $h$.
\end{lemma}

\begin{proof}
We know that $f\ra{ (f_0,s) }g$ and $g\ra{(g_0,s')} h$. Therefore, by definition:
\begin{align*}
&h_{0}(r)=f_{0}(r)+(\partial ^{\prime }\circ (s+s^{\prime }))(r), && h_{1}(e)=f_{1}(e)+((s+s^{\prime })\circ \partial )(e).
\end{align*}%
Let us see that $s+s^{\prime }$ satisfies the condition for it to be an $f_0$ derivation:
\begin{align*}
(s+&s^{\prime })(rr^{\prime })  = s(rr^{\prime })+s^{\prime }(rr^{\prime })
\\
& = f_{0}(r)\blacktriangleright s(r^{\prime })+f_{0}(r^{\prime
})\blacktriangleright s(r)+s(r)s(r^{\prime })+g_{0}(r)\blacktriangleright
s^{\prime }(r^{\prime }) \\
& \quad \quad \quad  +g_{0}(r^{\prime })\blacktriangleright s^{\prime
}(r)+s^{\prime }(r)s^{\prime }(r^{\prime }) \\
& = f_{0}(r)\blacktriangleright s(r^{\prime })+f_{0}(r^{\prime
})\blacktriangleright s(r)+s(r)s(r^{\prime })+(f_{0}(r)+(\partial ^{\prime
}\circ s)(r))\blacktriangleright s^{\prime }(r^{\prime }) \\
&   \quad \quad \quad +(f_{0}(r^{\prime })+(\partial ^{\prime }\circ s)(r^{\prime
}))\blacktriangleright s^{\prime }(r)+s^{\prime }(r)s^{\prime }(r^{\prime })
\\
& = f_{0}(r)\blacktriangleright s(r^{\prime })+f_{0}(r^{\prime
})\blacktriangleright s(r)+s(r)s(r^{\prime })+f_{0}(r)\blacktriangleright
s^{\prime }(r^{\prime })+(\partial ^{\prime }\circ s)(r)\blacktriangleright
s^{\prime }(r^{\prime })
\\
&  \quad \quad \quad   +f_{0}(r^{\prime })\blacktriangleright s^{\prime }(r)+(\partial ^{\prime }\circ s)(r^{\prime })\blacktriangleright s^{\prime
}(r)+s^{\prime }(r)s^{\prime }(r^{\prime }) \\
& = f_{0}(r)\blacktriangleright s(r^{\prime })+f_{0}(r^{\prime
})\blacktriangleright s(r)+s(r)s(r^{\prime })+f_{0}(r)\blacktriangleright
s^{\prime }(r^{\prime })+s(r)s^{\prime }(r^{\prime }) \\
&   \quad   +f_{0}(r^{\prime
})\blacktriangleright s^{\prime }(r)+s(r^{\prime })s^{\prime }(r)+s^{\prime }(r)s^{\prime }(r^{\prime }) \quad \bf{(\because 2^{nd} \textbf{ {Peiffer-Whitehead} law})} \\
& = f_{0}(r)\blacktriangleright (s+s^{\prime })(r^{\prime })+f_{0}(r^{\prime
})\blacktriangleright (s+s^{\prime })(r)+(s+s^{\prime })(r)(s+s^{\prime
})(r^{\prime }),
\end{align*}
for all $r,r^{\prime }\in R$. Therefore $(s+s')$ is an $f_{0}$-derivation connecting $f$ to $h$.
\end{proof}

\begin{remark}\label{PeifferNeeded}{
Note that we explicitly used the second {Peiffer-Whitehead} relation for crossed modules of algebras in the previous two proofs. Therefore, these results are not true in the pre-crossed module and 2-crossed module cases. Most of the discussion we  will present on the homotopy of 2-crossed module maps is a way {to solve} this issue, in the particular case when the domain 2-crossed module is free up to order one. }
\end{remark}

\begin{theorem}{Let $\mathcal{A}$ and $\mathcal{A^{\prime}}$ be two arbitrary crossed modules of commutative algebras. We have a groupoid ${\rm HOM}(\A,\A')$,  whose objects are the crossed module maps $\mathcal{%
A\rightarrow A}^{\prime }$,  the  morphisms being their homotopies. In particular the  relation below, for crossed module morphisms $\A \to \A'$, is an equivalence relation:}
\begin{equation}
\text{\textquotedblleft }f\simeq g\Longleftrightarrow \text{there exists an }%
f_{0}\text{-derivation }s \text{ connecting } f \text{ with } g\text{\textquotedblright }.
\end{equation}
\end{theorem}

\section{Pointed Homotopy of 2-Crossed Module Maps}

\noindent Now fix  2-crossed modules $\mathcal{A}%
=(L,E,R,\partial _{1},\partial _{2},\{,\})$ and $\mathcal{A}^{\prime }=(L^{\prime
},E^{\prime },R^{\prime },\partial _{1}^{\prime },\partial _{2}^{\prime },\{,\})$.

\begin{definition}\label{qder}
Let $f\colon \mathcal{A\to A}^{\prime }$ be a 2-crossed module
morphism. A quadratic $f$-derivation is a pair
$(s,t)$,
where $s\colon R \to E'$, $t\colon E \to L'$ are $\k$-linear maps, satisfying:
\begin{align}
s(rr^{\prime })= f_{0}(r)\blacktriangleright s(r^{\prime})+f_{0}(r^{\prime })\blacktriangleright s(r)+s(r)s(r^{\prime }),
\end{align}
{(which means that $s\colon R \to E'$ is an $f_{0}$-derivation)} and, for all $r \in R$ and $e, e' \in E$:
\begin{multline}
t(ee^{\prime })  =  \{(s\circ \partial _{1})(e)\otimes
f_{1}(e^{\prime })\}+\{(s\circ \partial _{1})(e^{\prime })\otimes
f_{1}(e)\}+f_{1}(e)\blacktriangleright ^{\prime }t(e^{\prime
})\\+f_{1}(e^{\prime })\blacktriangleright ^{\prime }t(e)+(s\circ \partial _{1})(e)\blacktriangleright ^{\prime }t(e^{\prime
})+(s\circ \partial _{1})(e^{\prime })\blacktriangleright ^{\prime
}t(e)+t(e)t(e^{\prime }),
\end{multline}
\begin{multline}
t(r\blacktriangleright e)= f_{0}(r)\blacktriangleright
t(e)+(\partial _{1}^{\prime }\circ s)(r)\blacktriangleright
t(e)+\{s(r)\otimes f_{1}(e)\}\\-\{f_{1}(e)\otimes s(r)\}-\{(s\circ \partial
_{1})(e)\otimes s(r)\}.
\end{multline}
\end{definition}

\begin{lemma}\label{simplify}
If $(s,t)$ is a quadratic $f$-derivation, then, for all $l,l' \in L$, $r \in R$:
\begin{align*}
t\big(\partial_{2} (l)\,\partial_{2} (l^{\prime })\big)&=f_{2}(l)(t\circ\partial_{2} )(l^{\prime
})+f_{2}(l^{\prime })(t\circ\partial_{2} )(l)+(t\circ\partial_{2} )(l)(t\circ\partial_{2})(l^{\prime }),
\\
t\big (r\blacktriangleright \partial_{2} (l)\big)&=f_{0}(r)\blacktriangleright (t\circ\partial_{2}
)(l)+(\partial_{1} ^{\prime }\circ s)(r)\blacktriangleright f_{2}(l)+(\partial_{1} ^{\prime
}\circ s)(r)\blacktriangleright (t\circ\partial_{2} )(l).
\end{align*}
\end{lemma}

\begin{proof}
{To make the formulae more compact, in this proof (and others), we do not use the ``$\circ$'' to denote composition, and put} $\{e,f\}$, instead of  $\{e \otimes f\}$, for Peiffer liftings.
{By using the  second {Peiffer-Whitehead} law,  since  $(\d_2\colon L \to E,\t')$ is a crossed module:}
\begin{align*}
t\big(\partial_{2}& (l)\,\partial_{2} (l^{\prime })\big)\\ & = \{(s\partial_{1} )(\partial_{2} (l)),f_{1}
(\partial_{2} (l^{\prime }))\}+\{(s\partial_{1} )(\partial_{2} (l^{\prime })),f_{1} (\partial_{2}
(l))\} \\
& \quad +f_{1} (\partial_{2} (l))\blacktriangleright' t(\partial_{2} (l^{\prime })) +f_{1}
(\partial_{2} (l^{\prime }))\blacktriangleright' t(\partial_{2} (l)) \\
& \quad \quad +(s\partial_{1} )(\partial_{2} (l))\blacktriangleright' t(\partial_{2} (l^{\prime
}))+(s\partial_{1} )(\partial_{2} (l^{\prime }))\blacktriangleright' t(\partial_{2} (l))+t(\partial_{2}(l))t(\partial_{2} (l^{\prime })) \\
& = \{s((\partial_{1} \partial_{2} )(l)),(f_{1} \partial_{2} )(l^{\prime })\}+\{s((\partial_{1}
\partial_{2} )(l^{\prime })),(f_{1} \partial_{2} )(l)\} \\
& \quad +(f_{1} \partial_{2}
)(l)\blacktriangleright' (t\partial_{2} )(l^{\prime })+(f_{1} \partial_{2} )(l^{\prime
})\blacktriangleright' (t\partial_{2} )(l) \\
& \quad \quad +s((\partial_{1} \partial_{2} )(l))\blacktriangleright' (t\partial_{2} )(l^{\prime
})+s((\partial_{1} \partial_{2} )(l^{\prime }))\blacktriangleright' (t\partial_{2} )(l)+(t\partial_{2})(l)(t\partial_{2} )(l^{\prime }) \\
& = \{s(0_{G}),(f_{1} \partial_{2} )(l^{\prime })\}+\{s(0_{G}),(f_{1} \partial_{2}
)(l)\} \\
& \quad +(\partial_{2} ^{\prime }f_{2} )(l)\blacktriangleright' (t\partial_{2} )(l^{\prime
})+(\partial_{2} ^{\prime }f_{2} )(l^{\prime })\blacktriangleright' (t\partial_{2} )(l) \\
& \quad \quad +s(0_{G})\blacktriangleright' (t\partial_{2} )(l^{\prime
})+s(0_{G})\blacktriangleright' (t\partial_{2} )(l) +(t\partial_{2} )(l)(t\partial_{2} )(l^{\prime}) \\
& = \{0_{E^{\prime }},(f_{1} \partial_{2} )(l^{\prime })\}+\{0_{E^{\prime
}},(f_{1} \partial_{2} )(l)\} \\
& \quad +\partial_{2} ^{\prime }(f_{2} (l))\blacktriangleright' (t\partial_{2}
)(l^{\prime })+\partial_{2} ^{\prime }(f_{2} (l^{\prime }))\blacktriangleright'
(t\partial_{2} )(l) \\
& \quad \quad +0_{E^{\prime }}\blacktriangleright' (t\partial_{2} )(l^{\prime })+0_{E^{\prime
}}\blacktriangleright' (t\partial_{2} )(l)+(t\partial_{2} )(l)(t\partial_{2} )(l^{\prime }) \\
& = f_{2} (l)(t\partial_{2} )(l^{\prime })+f_{2} (l^{\prime })(t\partial_{2}
)(l)+(t\partial_{2} )(l)(t\partial_{2} )(l^{\prime }),
\end{align*}
for all $l,l^{\prime} \in L$. Also:
\begin{align*}
t(r\blacktriangleright \partial_{2} (l)) & = f_{0} (r)\blacktriangleright
t(\partial_{2} (l))+(\partial_{1} ^{\prime }s)(r)\blacktriangleright t(\partial_{2}
(l))+\{s(r),f_{1} (\partial_{2} (l))\} \\
& \quad  \quad-\{f_{1} (\partial_{2} (l)),s(r)\}-\{(s\partial_{1} )(\partial_{2} (l)),s(r)\} \\
& = f_{0} (r)\blacktriangleright (t\partial_{2} )(l)+(\partial_{1} ^{\prime
}s)(r)\blacktriangleright (t\partial_{2} )(l)+\{s(r),(f_{1} \partial_{2} )(l)\} \\
& \quad  \quad-\{(f_{1}\partial_{2} )(l),s(r)\}-\{(s(\partial_{1} \partial_{2} )(l)),s(r)\} \\
& = f_{0} (r)\blacktriangleright (t\partial_{2} )(l)+(\partial_{1} ^{\prime
}s)(r)\blacktriangleright (t\partial_{2} )(l)+\{s(r),(\partial_{2} ^{\prime }f_{2}
)(l)\} \\
& \quad  \quad-\{(\partial_{2} ^{\prime }f_{2} )(l),s(r)\}-\{(s(0_{G}),s(r)\} \\
& = f_{0} (r)\blacktriangleright (t\partial_{2} )(l)+(\partial_{1} ^{\prime
}s)(r)\blacktriangleright (t\partial_{2} )(l)+\{s(r),\partial_{2} ^{\prime }(f_{2}
(l))\} \\
& \quad  \quad \quad-\{\partial_{2} ^{\prime }(f_{2} (l)),s(r)\}-\{0_{E^{\prime }},s(r)\} \\
& = f_{0} (r)\blacktriangleright (t\partial_{2} )(l)+(\partial_{1} ^{\prime
}s)(r)\blacktriangleright (t\partial_{2} )(l)+s(r)\blacktriangleright ^{\prime }f_{2}
(l) \\
& \quad  \quad-[s(r)\blacktriangleright ^{\prime }f_{2} (l)-\partial_{1} ^{\prime
}(s(r))\blacktriangleright f_{2} (l)] \\
& = f_{0} (r)\blacktriangleright (t\partial_{2} )(l)+(\partial_{1} ^{\prime
}s)(r)\blacktriangleright (t\partial_{2} )(l)+(\partial_{1} ^{\prime
}s)(r)\blacktriangleright f_{2} (l).
\end{align*}
for all $l \in L$ and $r\in R$.
\end{proof}

\begin{theorem}
[Pointed homotopy of 2-crossed module maps]\label{ph2cmm} Let $f=(f_2,f_1,f_0)$ be a 2-crossed
module morphism  $\mathcal{A\to A}^{\prime }$. In the
condition of the previous definition, if $(s,t)$ is a quadratic $f$%
-derivation, and if we define $g=(g_{2},g_{1},g_{0})$ as:
\begin{equation}\label{sourcetarget}
\begin{split}
g_{0}(r) & = f_{0}(r)+(\partial _{1}^{\prime } \circ s)(r) ,\\
g_{1}(e) & = f_{1}(e)+(s \circ \partial _{1})(e)+(\partial _{2}^{\prime} \circ t)(e) ,\\
g_{2}(l) & = f_{2}(l)+(t \circ \partial _{2})(l),
\end{split}
\end{equation}
where $r\in R,\ e\in E$ and $l\in L$, then $g$ also defines a 2-crossed
module map  $\mathcal{A\to A}^{\prime }$. In such
case, we use the notation:
\begin{equation*}
f\ra{(f,s,t)}g,
\end{equation*}
and say that $(f,s,t)$ is a homotopy (or quadratic derivation), connecting $f$ to $g$.
\end{theorem}


\begin{proof} {Firstly we show that $g_{0}$, $g_{1}$ and $g_{2}$ define algebra morphisms. That $g_{0}(r+r^{\prime })=g_{0}(r)+g_{0}(r^{\prime })$ and also $g_{0}(kr)=kg_{0}(r)$, follows
from $\k$-linearity, and similarly for $g_{1}$ and $g_{2}$. Also, using the first {Peiffer-Whitehead} relation in the penultimate step:}
\begin{align*}
g_{0}(rr^{\prime }) & =  f_{0} (rr^{\prime })+(\partial_{1} ^{\prime
}s)(rr^{\prime }) =  f_{0} (r)f_{0} (r^{\prime })+(\partial_{1} ^{\prime }(s(rr^{\prime })) \\
& =  f_{0} (r)f_{0} (r^{\prime })+\partial_{1} ^{\prime }[f_{0} (r)\blacktriangleright
s(r^{\prime })+f_{0} (r^{\prime })\blacktriangleright s(r)+s(r)s(r^{\prime })]
\\
& =  f_{0} (r)f_{0} (r^{\prime })+\partial_{1} ^{\prime }(f_{0} (r)\blacktriangleright
s(r^{\prime }))+\partial_{1} ^{\prime }(f_{0} (r^{\prime })\blacktriangleright
s(r))+\partial_{1} ^{\prime }(s(r)s(r^{\prime }))  \\
& =  f_{0} (r)f_{0} (r^{\prime })+f_{0} (r)(\partial_{1} ^{\prime }(s(r^{\prime
}))+f_{0} (r^{\prime })(\partial_{1} ^{\prime }(s(r))+\partial_{1} ^{\prime }(s(r))\partial_{1}
^{\prime }(s(r^{\prime }))  \\
& =  [f_{0} (r)+(\partial_{1} ^{\prime }s)(r)][f_{0} (r^{\prime })+(\partial_{1} ^{\prime
}s)(r^{\prime })] 
=  g_{0}(r)g_{0}(r^{\prime }),
\end{align*}
for all $r,r^{\prime} \in R$ and $k \in \k$, which means $g_{0}$ is an algebra morphism. Similarly:
\begin{align*}&g_{1}(ee^{\prime }) =f_{1} (ee^{\prime })+(s\partial_{1}
)(ee^{\prime })+(\partial_{2} ^{\prime }t)(ee^{\prime }) \\
&=f_{1} (e)f_{1} (e^{\prime })+s(\partial_{1} (ee^{\prime }))+\partial_{2} ^{\prime
}(t(ee^{\prime })) = f_{1} (e)f_{1} (e^{\prime })+s(\partial_{1} (e)\partial_{1} (e^{\prime }))+\partial_{2}
^{\prime }(t(ee^{\prime })) \\
&=f_{1} (e)f_{1} (e^{\prime })+f_{0} (\partial_{1} (e))\blacktriangleright s(\partial_{1}
(e^{\prime }))+f_{0} (\partial_{1} (e^{\prime }))\blacktriangleright s(\partial_{1}
(e))+s(\partial_{1} (e))s(\partial_{1} (e^{\prime })) \\
& \quad  +\partial_{2} ^{\prime }[\{(s\partial_{1} )(e),f_{1} (e^{\prime })\}+\{(s\partial_{1}
)(e^{\prime }),f_{1} (e)\}+f_{1} (e)\blacktriangleright t(e^{\prime })+f_{1}
(e^{\prime })\blacktriangleright t(e) \\
& \quad \quad +(s\partial_{1} )(e)\blacktriangleright t(e^{\prime })+(s\partial_{1} )(e^{\prime
})\blacktriangleright t(e)+t(e)t(e^{\prime })] \\
&=f_{1} (e)f_{1} (e^{\prime })+(f_{0} \partial_{1} )(e)\blacktriangleright (s\partial_{1}
)(e^{\prime })+(f_{0} \partial_{1} )(e^{\prime })\blacktriangleright (s\partial_{1}
)(e)+(s\partial_{1} )(e)(s\partial_{1} )(e^{\prime }) \\
& \quad +\partial_{2} ^{\prime }\{(s\partial_{1} )(e),f_{1} (e^{\prime })\}+\partial_{2} ^{\prime
}\{(s\partial_{1} )(e^{\prime }),f_{1} (e)\}+\partial_{2} ^{\prime }(f_{1}
(e)\blacktriangleright t(e^{\prime })) \\
& \quad \quad +\partial_{2} ^{\prime }(f_{1} (e^{\prime
})\blacktriangleright t(e))+\partial_{2} ^{\prime }((s\partial_{1} )(e)\blacktriangleright t(e^{\prime
}))+\partial_{2} ^{\prime }((s\partial_{1} )(e^{\prime })\blacktriangleright t(e))+\partial_{2}
^{\prime }(t(e)t(e^{\prime })) \\
&=f_{1} (e)f_{1} (e^{\prime })+(\partial_{1} ^{\prime }f_{1}
)(e)\blacktriangleright (s\partial_{1} )(e^{\prime })+(\partial_{1} ^{\prime }f_{1}
)(e^{\prime })\blacktriangleright (s\partial_{1} )(e)+(s\partial_{1} )(e)(s\partial_{1}
)(e^{\prime }) \\
& \quad +(s\partial_{1} )(e)f_{1} (e^{\prime })-\partial_{1} ^{\prime }(f_{1} (e^{\prime
}))\blacktriangleright (s\partial_{1} )(e)+(s\partial_{1} )(e^{\prime })f_{1} (e)-\partial_{1}
^{\prime }(f_{1} (e))\blacktriangleright (s\partial_{1} )(e^{\prime }) \\
& \quad \quad +f_{1} (e)\partial_{2} ^{\prime }(t(e^{\prime }))+f_{1} (e^{\prime })\partial_{2}
^{\prime }(t(e))+(s\partial_{1} )(e)\partial_{2} ^{\prime }(t(e^{\prime }))+(s\partial_{1}
)(e^{\prime })\partial_{2} ^{\prime }(t(e)) \\
& \quad \quad \quad +\partial_{2} ^{\prime }(t(e))\partial_{2} ^{\prime
}(t(e^{\prime })) \\
&=f_{1} (e)f_{1} (e^{\prime })+\partial_{1} ^{\prime }(f_{1}
(e))\blacktriangleright (s\partial_{1} )(e^{\prime })+\partial_{1} ^{\prime }(f_{1}
(e^{\prime }))\blacktriangleright (s\partial_{1} )(e)+(s\partial_{1} )(e)(s\partial_{1}
)(e^{\prime }) \\
& \quad +(s\partial_{1} )(e)f_{1} (e^{\prime })-\partial_{1} ^{\prime }(f_{1} (e^{\prime
}))\blacktriangleright (s\partial_{1} )(e)+(s\partial_{1} )(e^{\prime })f_{1} (e)-\partial_{1}
^{\prime }(f_{1} (e))\blacktriangleright (s\partial_{1} )(e^{\prime }) \\
& \quad \quad +f_{1} (e)(\partial_{2} ^{\prime }t)(e^{\prime })+f_{1} (e^{\prime })(\partial_{2}
^{\prime }t)(e)+(s\partial_{1} )(e)(\partial_{2} ^{\prime }t)(e^{\prime })+(s\partial_{1}
)(e^{\prime })\partial_{2} ^{\prime }(t(e)) \\
& \quad \quad \quad +(\partial_{2} ^{\prime }t)(e)(\partial_{2}
^{\prime }t)(e^{\prime }) \\
&=f_{1} (e)f_{1} (e^{\prime })+(s\partial_{1} )(e)(s\partial_{1} )(e^{\prime
})+(s\partial_{1} )(e)f_{1} (e^{\prime })+(s\partial_{1} )(e^{\prime })f_{1} (e)+f_{1}
(e)(\partial_{2} ^{\prime }t)(e^{\prime }) \\
& \quad  +f_{1} (e^{\prime })(\partial_{2} ^{\prime }t)(e)+(s\partial_{1} )(e)(\partial_{2}
^{\prime }t)(e^{\prime })+(s\partial_{1} )(e^{\prime })\partial_{2} ^{\prime
}(t(e))+(\partial_{2} ^{\prime }t)(e)(\partial_{2} ^{\prime }t)(e^{\prime }) \\
&=[f_{1} (e)+(s\partial_{1} )(e)+(\partial_{2} ^{\prime }t)(e)][f_{1} (e^{\prime
})+(s\partial_{1} )(e^{\prime })+(\partial_{2} ^{\prime }t)(e^{\prime })] = g_{1}(e)g_{1}(e^{\prime }),
\end{align*}
for all $e,e^{\prime} \in E$, thus  $g_{1}$ is an algebra morphism. By using Lemma \ref{simplify}, we have:
\begin{align*}
g_{2}(ll^{\prime }) &=f_{2} (ll^{\prime })+(t\partial_{2}
)(ll^{\prime }) =f_{2} (l)f_{2} (l^{\prime })+t(\partial_{2} (ll^{\prime }))=f_{2} (l)f_{2} (l^{\prime })+t(\partial_{2} (l)\partial_{2} (l^{\prime })) \\
&=f_{2} (l)f_{2} (l^{\prime })+f_{2} (l)((t\partial_{2} )(l^{\prime }))+f_{2}
(l^{\prime })((t\partial_{2} )(l))+(t\partial_{2} )(l)(t\partial_{2} )(l^{\prime }) \\
&=[f_{2} (l)+(t\partial_{2} )(l)][f_{2} (l^{\prime })+(t\partial_{2} )(l^{\prime })] =g_{2}(l)g_{2}(l^{\prime }),
\end{align*}
for all $l,l^{\prime} \in L$, thus $g_{2}$ is an algebra morphism. Clearly the diagram below commutes:
$$ \xymatrix@R=25pt@C=25pt{
L
\ar[r]^{\partial_{2}}
\ar[d]_{g_{2}}
& E
\ar[r]^{\partial_{1}}
\ar[d]_{g_{1}}
& R
\ar[d]^{g_{0}}
\\
L^\prime
\ar[r]_{{\partial}^{\prime} _{2} }
& E^\prime
\ar[r]_{{\partial}^{\prime} _{1} }
& R^\prime
}$$
We now show that these morphisms preserve the actions and the Peiffer liftings:%
\begin{align*}
&g_{1}(r\blacktriangleright e) = f_{1} (r\blacktriangleright
e)+(s\partial_{1} )(r\blacktriangleright e)+(\partial_{2} ^{\prime }t)(r\blacktriangleright
e) \\
&= f_{0} (r)\blacktriangleright f_{1} (e)+s(\partial_{1} (r\blacktriangleright
e))+\partial_{2} ^{\prime }(t(r\blacktriangleright e)) \\
&= f_{0} (r)\blacktriangleright f_{1} (e)+s(r\partial_{1} (e))+\partial_{2} ^{\prime
}(t(r\blacktriangleright e)) \\
&= f_{0} (r)\blacktriangleright f_{1} (e)+f_{0} (r)\blacktriangleright s(\partial_{1}
(e))+f_{0} (\partial_{1} (e))\blacktriangleright s(r)+s(r)s(\partial_{1} (e)) \\
&\quad +\partial_{2} ^{\prime }[f_{0} (r)\blacktriangleright t(e)+(\partial_{1} ^{\prime
}s)(r)\blacktriangleright t(e)+\{s(r),f_{1} (e)\}-\{f_{1} (e),s(r)\} \\
&\quad \quad -\{(s\partial_{1})(e),s(r)\} \\
&= f_{0} (r)\blacktriangleright f_{1} (e)+f_{0} (r)\blacktriangleright (s\partial_{1}
)(e)+(f_{0} \partial_{1} )(e)\blacktriangleright s(r)+s(r)(s\partial_{1} )(e) \\
&\quad \quad +\partial_{2}^{\prime }(f_{0} (r)\blacktriangleright t(e))+\partial_{2} ^{\prime }((\partial_{1} ^{\prime }s)(r)\blacktriangleright
t(e))+\partial_{2} ^{\prime }\{s(r),f_{1} (e)\} \\
&\quad \quad \quad -\partial_{2} ^{\prime }\{f_{1}
(e),s(r)\}-\partial_{2} ^{\prime }\{(s\partial_{1} )(e),s(r)\} \\
&= f_{0} (r)\blacktriangleright f_{1} (e)+f_{0} (r)\blacktriangleright (s\partial_{1}
)(e)+(\partial_{1} ^{\prime }f_{1} )(e)\blacktriangleright s(r)+s(r)(s\partial_{1}
)(e) \\
&\quad +f_{0} (r)\blacktriangleright \partial_{2} ^{\prime }(t(e))+(\partial_{1} ^{\prime }s)(r)\blacktriangleright \partial_{2} ^{\prime
}(t(e))+s(r)f_{1} (e)-\partial_{1} ^{\prime }(f_{1} (e))\blacktriangleright s(r) \\
& \quad \quad -f_{1} (e)s(r)+\partial_{1} ^{\prime }(s(r))\blacktriangleright f_{1} (e)-(s\partial_{1} )(e)s(r)+(\partial_{1} 
^{\prime }s)(r)\blacktriangleright (s\partial_{1})(e) \\
&= f_{0} (r)\blacktriangleright f_{1} (e)+f_{0} (r)\blacktriangleright (s\partial_{1}
)(e)+f_{0} (r)\blacktriangleright \partial_{2} ^{\prime }(t(e))+(\partial_{1} ^{\prime
}s)(r)\blacktriangleright \partial_{2} ^{\prime }(t(e)) \\
&\quad +\partial_{1} ^{\prime }(s(r))\blacktriangleright f_{1} (e)+(\partial_{1} ^{\prime
}s)(r)\blacktriangleright (s\partial_{1} )(e) \\
&= f_{0} (r)\blacktriangleright f_{1} (e)+f_{0} (r)\blacktriangleright (s\partial_{1}
)(e)+f_{0} (r)\blacktriangleright (\partial_{2} ^{\prime }t)(e)+(\partial_{1} ^{\prime
}s)(r)\blacktriangleright \partial_{2} ^{\prime }(t(e)) \\
&\quad +(\partial_{1} ^{\prime }s)(r)\blacktriangleright f_{1} (e)+(\partial_{1} ^{\prime
}s)(r)\blacktriangleright (s\partial_{1} )(e) \\
&= [f_{0} (r)+(\partial_{1} ^{\prime }s)(r)]\blacktriangleright \lbrack f_{1}
(e)+(s\partial_{1} )(e)+(\partial_{2} ^{\prime }t)(e)] = g_{0}(r)\blacktriangleright g_{1}(e).
\end{align*}
Again, by using Lemma \ref{simplify} we have:%
\begin{align*}
g_{2}(r\blacktriangleright l) &= f_{2} (r\blacktriangleright
l)+(t\partial_{2} )(r\blacktriangleright l) = f_{0} (r)\blacktriangleright f_{2} (l)+t(\partial_{2}(r\blacktriangleright l)) \\
&= f_{0} (r)\blacktriangleright f_{2} (l)+t(r\blacktriangleright \partial_{2} (l)) \\
&= f_{0} (r)\blacktriangleright f_{2} (l)+f_{0} (r)\blacktriangleright (t\partial_{2}
)(l)+(\partial_{1} ^{\prime }s)(r)\blacktriangleright f_{2} (l)+(\partial_{1} ^{\prime
}s)(r)\blacktriangleright (t\partial_{2} )(l) \\
&= f_{0} (r)\blacktriangleright (f_{2} (l)+(t\partial_{2} )(l))+(\partial_{1} ^{\prime
}s)(r)\blacktriangleright (f_{2} (l)+(t\partial_{2} )(l)) \\
&= [f_{0} (r)+(\partial_{1} ^{\prime }s)(r)]\blacktriangleright \lbrack f_{2}
(l)+(t\partial_{2} )(l)] = g_{0}(r)\blacktriangleright g_{2}(l).
\end{align*}
Finally:
\begin{align*}
&g_{2}(\{e,e^{\prime }\}) = f_{2} (\{e,e^{\prime }\})+(t\partial_{2}
)(\{e,e^{\prime }\}) = \{f_{1} (e),f_{1} (e^{\prime })\}+t(\partial_{2} \{e,e^{\prime }\}) \\
&= \{f_{1} (e),f_{1} (e^{\prime })\}+t(ee^{\prime }-\partial_{1} (e^{\prime
})\blacktriangleright e) = \{f_{1} (e),f_{1} (e^{\prime })\}+t(ee^{\prime })-t(\partial_{1} (e^{\prime
})\blacktriangleright e) \\
&= \{f_{1} (e),f_{1} (e^{\prime })\}+\{(s\partial_{1} )(e),f_{1} (e^{\prime
})\}+\{(s\partial_{1} )(e^{\prime }),f_{1} (e)\}+f_{1} (e)\blacktriangleright
t(e^{\prime }) \\
&\quad +f_{1} (e^{\prime })\blacktriangleright t(e)+(s\partial_{1} )(e)\blacktriangleright t(e^{\prime })+(s\partial_{1} )(e^{\prime
})\blacktriangleright t(e)+t(e)t(e^{\prime }) \\
&\quad  -f_{0} (\partial_{1} (e^{\prime
}))\blacktriangleright t(e)-(\partial_{1} ^{\prime }s)(\partial_{1} (e^{\prime
}))\blacktriangleright t(e)-\{s(\partial_{1} (e^{\prime })),f_{1} (e)\}+\{f_{1} (e),s(\partial_{1} (e^{\prime
}))\} \\
&\quad+\{(s\partial_{1} )(e),s(\partial_{1} (e^{\prime }))\} \\
&= \{f_{1} (e),f_{1} (e^{\prime })\}+\{(s\partial_{1} )(e),f_{1} (e^{\prime
})\}+\{(s\partial_{1} )(e^{\prime }),f_{1} (e)\}+\{f_{1} (e),\partial_{2} ^{\prime
}(t(e^{\prime }))\} \\
&\quad +f_{1} (e^{\prime })\blacktriangleright t(e)+\{(s\partial_{1} )(e),\partial_{2} ^{\prime }(t(e^{\prime }))\}+(s\partial_{1}
)(e^{\prime })\blacktriangleright t(e)+\{\partial_{2} ^{\prime }(t(e)),\partial_{2}
^{\prime }(t(e^{\prime }))\} \\
&\quad  -\partial_{1} ^{\prime }(f_{1} (e^{\prime
}))\blacktriangleright t(e)-(\partial_{1} ^{\prime }s)(\partial_{1} (e^{\prime }))\blacktriangleright
t(e)-\{s(\partial_{1} (e^{\prime })),f_{1} (e)\}+\{f_{1} (e),s(\partial_{1} (e^{\prime
}))\} \\
&\quad +\{(s\partial_{1} )(e),s(\partial_{1} (e^{\prime }))\} \\
&= \{f_{1} (e),f_{1} (e^{\prime })\}+\{f_{1} (e),(s\partial_{1} )(e^{\prime
})\}+\{f_{1} (e),(\partial_{2} ^{\prime }t)(e^{\prime }))\} \\
&\quad +\{(s\partial_{1} )(e),f_{1} (e^{\prime })\}+\{(s\partial_{1} )(e),(s\partial_{1}
)(e^{\prime })\}+\{(s\partial_{1} )(e),(\partial_{2} ^{\prime }t)(e^{\prime })\} \\
&\quad  +\{(\partial_{2} ^{\prime }t)(e),f_{1} (e^{\prime })\}+\{(\partial_{2} ^{\prime
}t)(e),(s\partial_{1} )(e^{\prime })\}+\{(\partial_{2} ^{\prime }t)(e),(\partial_{2} ^{\prime
}t)(e^{\prime })\} \\
&= \{f_{1} (e)+(s\partial_{1} )(e)+(\partial_{2} ^{\prime }t)(e),f_{1} (e^{\prime
})+(s\partial_{1} )(e^{\prime })+(\partial_{2} ^{\prime }t)(e^{\prime })\} = \{g_{1}(e),g_{1}(e^{\prime })\},
\end{align*}
for all $l\in L,\ e,e^{\prime}\in E$ and $r\in R$ which completes the proof.
\end{proof}

\subsection{Groupoid structure for 2-crossed module maps and their homotopies}

\begin{lemma}
Consider a 2-crossed module map $f\colon \mathcal{A\to A}%
^{\prime }$. Then the pair $
(0_{s},0_{t})
$, where each component is the  zero map, is  a quadratic $f$-derivation connecting $f$ to $f$.
\end{lemma}

\subsubsection{Crossed module homotopy versus  2-crossed module homotopy}
\noindent Recall the construction of the groupoid  with objects the crossed module maps between two crossed modules, {the morphisms being their homotopies;} Subsection \ref{homotopies}.

\noindent\begin{idea}
Let $f \colon\mathcal{A\to A}^{\prime }$ be a 2-crossed module
morphism and $(s,t)$ be a quadratic $f$-derivation, connecting $f$ to $g$. Then consider the tuple $(-s\colon R \to E',-t\colon E \to L').$
\end{idea}

\medskip

\noindent\begin{problem}
$(-s,-t)$ is not necessarily a quadratic $g$-derivation. In fact not even $-s$ is necessarily a derivation; see Remark \ref{PeifferNeeded}. 
\end{problem}

\medskip

\noindent\begin{idea}
Let $f,g$ and $h$ be 2-crossed module morphisms $\mathcal{%
A\to A}^{\prime }$. Let $(s,t)$ be a quadratic $f$-derivation connecting $f$
to $g$, and $(s^{\prime },t^{\prime })$ be a quadratic $g$-derivation connecting $g$
to $h$.
Since $f\ra{(f,s,t)} g$ and $g \ra{(g,s',t')} h$, we have, by Theorem \ref{ph2cmm}:
\begin{align*}
h_{0}(r) & = f_{0}(r)+(\partial _{1}^{\prime } \circ (s+s^{\prime }))(r), \\
h_{1}(e) & = f_{1}(e)+((s+s^{\prime }) \circ \partial _{1})(e)+(\partial
_{2}^{\prime } \circ (t+t^{\prime }))(e), \\
h_{2}(l) & = f_{2}(l)+((t+t^{\prime }) \circ \partial _{2})(l).
\end{align*}
\end{idea}

\noindent\begin{problem}
The map $(s+s')$ is not necessarily an $f_{0}$-derivation; see Remark \ref{PeifferNeeded}.
\end{problem}

\medskip

{Therefore,  the concatenation we used in the crossed module case (addition) is not a good binary
operation when composing homotopies of  2-crossed module maps. A similar issue happens for the inverse of quadratic derivations.}

\begin{warning}\label{warn} The previous observations tell us that, 
unlike the crossed module case, the homotopy relation between 2-crossed module maps $\A \to \A'$ is not an equivalence relation.
 This issue is shown clearly, by using an example, in \cite{GM1}.
From the model category theory point of view \cite{DS}, this is not  unexpected. The previous homotopy relation between 2-crossed module maps $\A \to \A'$ could be set by introducing a path-object 2-crossed module of $\A'$, and we would only expect homotopy to be an equivalence relation if $\A$ were cofibrant and $\A'$ fibrant (the latter likely always holds).
\end{warning}

\subsubsection{The main idea}
\begin{remark}
To overcome the problems just stated,  we now  work in a  more restricted context, discussing homotopy of 2-crossed modules maps $\A \to \A'$, in the case when {$\A=(L,E,R,\partial _{1},\partial _{2},\{,\})$ is free up to order one, with a chosen {(free commutative algebra)} basis $B$ of $R$. Therefore, $R$ is a polynomial algebra, with a formal variable assigned to each element of $B$; Definition \ref{freeness}.  {Presumably, these are the cofibrant objects of a yet to be discovered model category structure in the category of 2-crossed modules, very different from the one outlined in the introduction; see \cite{GM1}.} }
\end{remark}

\begin{lemma}\label{defphi}{
Looking at Definition \ref{semidirect} (semidirect products), if $f=(f_2,f_1,f_0)\colon \A \to \A'$, then $f_{0}$-derivations are in one-to-one correspondence with algebra maps, like:}
\begin{equation}
\phi : r \in  R \longmapsto \big(f_{0}(r),s(r)\big) \in R^{\prime }\ltimes _{\blacktriangleright
}E^{\prime }\doteq \A_1', \textrm{ where } r\in R.
\end{equation}
\end{lemma}

\begin{lemma}\label{uniquelyextended}
In particular, if $R$ is a free $\k$-algebra, over the set $B$, an $f_{0}$-derivation $%
s:R\to E^{\prime }$ can be specified (and uniquely) by its value
on $B \subset R$. Therefore, a set map $s^{\star}\colon B\to
E^{\prime }$ uniquely extends to an $f_{0}$-derivation $s$; see the diagram below:
\begin{equation}
\xymatrix@R=30pt@C=30pt{\free{B}{R}{\A'_1}{(f_{0},s^{\star})}{(f_{0},s)}}
\end{equation}
\end{lemma}
%
%

\subsubsection{Concatenation of homotopies}\label{concatenation}

\noindent Let $f,g,h\colon \A \to \A'$ be 2-crossed module maps; with $f=(f_2,f_1,f_0)$, $g=(g_2,g_1,g_0)$ and $h=(h_2,h_1,h_0)$. Let $(s,t)$ be a quadratic $f$-derivation connecting $f$
to $g$, and $(s^{\prime },t^{\prime })$ be a quadratic $g$-derivation connecting $g$
to $h$.
\begin{definition}
Let $s\boxplus s^{\prime }\colon R \to E'$ be the unique $f_{0}$-derivation with extends the restriction of the function $s+s^{\prime }$  to $B$; see the diagram below:
\begin{equation}
\xymatrix@R=30pt@C=30pt{\free{B}{R}{\A'_1}{(f_{0},s+s^{\prime })}{(f_{0},s\boxplus s^{\prime })}}
\end{equation}
Notice that, by definition,  for all $b\in B$, we have 
$(s\boxplus s^{\prime })(b)=(s+s^{\prime })(b).$
\end{definition}

\noindent Consider the set map $\zeta \colon B \to (R^{\prime }\ltimes _{\blacktriangleright
}E^{\prime })\blacktriangleright _{\bullet }(E^{\prime }\ltimes
_{\blacktriangleright ^{\prime }}L^{\prime })=\A'_2$, the algebra of 2-simplices of $\mathcal{A}^{\prime }=(L^{\prime},E^{\prime },R^{\prime },\partial _{1}^{\prime },\partial _{2}^{\prime })$, such that {(for all $b \in B$)}:
\begin{equation}
\zeta(b) =(f_{0}(b),s(b),s^{\prime }(b),0).\end{equation}
Geometrically:
\begin{align}
\hskip -1.55cm
\xymatrix{\\&b\in B\overset{\zeta }{\longmapsto } \phantom{--}}
\xymatrix@R=20pt@C=7pt{\trinearlyprime{0}{s'(b)}{s(b)}{f_{0}(b)}} 
{\xymatrix{\\&=\\}} \xymatrix@R=20pt@C=20pt{%
\trinearlyy{0}{f_{0}(b)}{g_{0}(b)}{h_{0}(b)}{s(b)}{s'(b)}{(s+s')(b)}}
\end{align}
By  definition of free algebra {(on the set $B$)}, there exists a unique algebra map:
\begin{equation*}
X^{(s,s^{\prime })}:R\longrightarrow (R^{\prime }\ltimes
_{\blacktriangleright }E^{\prime })\blacktriangleright _{\bullet }(E^{\prime
}\ltimes _{\blacktriangleright ^{\prime }}L^{\prime })=\A'_2
\end{equation*}%
extending $\zeta $; compare with the diagram below:
\begin{equation}
\xymatrix@R=30pt@C=30pt{\free{B}{R}{\A'_2}{\zeta}{X^{(s,s^{\prime })}}}
\end{equation}
By definition, for each $b\in B$, {we have} $X^{(s,s^{\prime })}(b)=(f_{0}(b),s(b),s^{\prime }(b),0)=\zeta (b) $. However, in general, if $r\in R$, then $
X^{(s,s^{\prime })}(r)\neq (f_{0}(r),s(r),s^{\prime }(r),0)=\zeta(r)$. 
Let us describe the element $X^{(s,s^{\prime })}(r)$ more explicitly. {By using the morphism $d_1$ of Remark \ref{b2}, together with Lemma \ref{uniquelyextended}, we conclude that $ X^{(s,s^{\prime })}(r)$ has the form:}
\begin{align}\label{defX}
{\xymatrix{ \\ X^{(s,s^{\prime })}(r)\quad =\quad }}\xymatrix@R=15pt@C=50pt{\trinearlyy{w^{(s,s')}(r)}{f_{0}(r)}{g_{0}(r)}{h_{0}(r)}{s(r)}{s^{\prime}(r)}{(s\boxplus s^\prime)(r)}}
\end{align}
for some uniquely defined map $w^{(s,s')}\colon R \to L'$.  Moreover:
\begin{lemma} {{The function} $w^{(s,s^{\prime })} \colon R\to L'$ is linear. Moreover,  {for each $ r,r' \in R$:}}
\begin{multline}\label{xsimp}
w^{(s,s')}(rr^{\prime })=f_{0}(r)\blacktriangleright w^{(s,s')}(r^{\prime
})+f_{0}(r^{\prime })\blacktriangleright w^{(s,s')}(r)+s^{\prime }(r^{\prime
})\blacktriangleright' w^{(s,s')}(r)
\\
+s(r^{\prime })\blacktriangleright' w^{(s,s')}(r)+s^{\prime
}(r)\blacktriangleright' w^{(s,s')}(r^{\prime })+s(r)\blacktriangleright' w^{(s,s')}(r^{\prime })
\\
-\{s^{\prime }(r^{\prime })\otimes s(r)\}-\{s^{\prime }(r)\otimes
s(r^{\prime })\}-w^{(s,s')}(r)w^{(s,s')}(r^{\prime }).
\end{multline}
{In addition $w^{(s,s')}$ measures the distance between $(s\boxplus s^{\prime })(r)$ and $(s+s^{\prime })(r)$, namely:}
\begin{equation}\label{wprop}
(s\boxplus s^{\prime })(r)=s(r)+s^{\prime }(r)-(\partial _{2}^{\prime }\circ w^{(s,s')})(r),\quad \quad  \textrm {for all } r \in R.
\end{equation}
\end{lemma}

\begin{proof}
The first assertion is a consequence of our convention for the semidirect product (Definition \ref{semidirect}), Lemma \ref{bullet} and the definition of the algebra of 2-simplices; equation \eqref{2simp}. On the other hand \eqref{wprop} follows  directly (in fact tautologically) from \eqref{2morsimp}.
\end{proof}

\begin{remark}
Note that $w^{(s,s')}(r)=0$, if $r \in B$, the free algebra basis of $R$.
\end{remark}

%
%
%

{If we pass  the triangle in \eqref{defX}  to the $n$-tuple notation, we  immediately see that:}
\begin{theorem}
There exists an algebra homomorphism of the form:
\begin{equation}
\begin{array}{lccl}
X^{(s,s^{\prime })}: & R & \longrightarrow  & (R^{\prime }\ltimes
_{\blacktriangleright }E^{\prime })\blacktriangleright _{\bullet }(E^{\prime
}\ltimes _{\blacktriangleright ^{\prime }}L^{\prime }) \\
& r & \longmapsto  & \big (f_{0}(r),s(r),s^{\prime }(r)-(\partial _{2}^{\prime
}\circ w^{(s,s')})(r),w^{(s,s')}(r)\big).
\end{array}%
\end{equation}
\end{theorem}
%

Now, we prove some additional, however crucial, properties of  $w^{(s,s')}\colon R \to L'$. 

%

\begin{lemma}\label{wzero}
If $s=0$ or $s^\prime=0$ then $w^{(s,s^{\prime })}=0$.
\end{lemma}

\begin{proof}
By the discussion above, we have an algebra map $X^{(s,s^{\prime })} \colon R \to \A'_2$, such that:
\begin{align}
{\xymatrix{ \\ X^{(s,s^{\prime })}(b)\quad =\quad }}\xymatrix@R=20pt@C=15pt{\trinearlyprime{0}{s^{\prime}(b)}{s(b)}{f_{0}(b)}}\xymatrix{\\\\ \quad ,}
\end{align}
{for all $b \in B$. If $s^{\prime}=0$, then we have the following element  $X^{(s,s^{\prime })}(b) \in \A_2'$:}
\begin{align}
{\xymatrix{ \\ X^{(s,s^{\prime })}(b)\quad =\quad }}\xymatrix@R=20pt@C=15pt{\trinearlyy{0}{f_{0}(b)}{f_{0}(b)+\partial'_{1}(s(b))}{f_{0}(b)+\partial'_{1}(s(b))}{s(b)}{0}{s(b)}}\xymatrix{\\\\ \quad ,}
\end{align}
which is also equal to (for notation, we refer to Remarks \ref{2degeneracies} and \ref{defphi}):
\begin{align}
s_{0}\left( {f_{0}(b)}\ra{s(b)} (f_0(b)+ \d'_1(s(b)) \right) & =  s_{0}(f_{0}(b),s(b)) =  (s_{0}\circ \phi)(b).
\end{align}
Therefore, $X^{(s,s^{\prime })}(b)=(s_{0}\circ \phi)(b)$, if  $b\in B$, the free basis of $R$. Since both $X^{(s,s^{\prime })}$ and  $(s_{0}\circ \phi)$  are algebra maps $R \to \A'_2$, follows that $X^{(s,s^{\prime })}(r)=(s_{0}\circ \phi)(r)$, for all $r\in R$. This in turn implies that $w^{(s,0)}(r)=0, \forall r \in R$, from \eqref{defX} and the visual form of $s_0$ in Remark \ref{2degeneracies}.  
%
By the same idea (using $s_1$ in Remark \ref{2degeneracies}), we prove that $w^{(0,s^{\prime })}=0$.
\end{proof}

\begin{lemma}\label{stnzero}
We have:
\begin{equation}
s\boxplus 0_{s}=s \text{  \ and \  } 0_{s'}\boxplus s'=s'.
\end{equation}
\end{lemma}

\begin{proof}
Follows from the definition of $\boxplus$ (since $s$, itself, is a derivation). It can also be deduced {from (and confirmed by)} the previous lemma, {together with  \eqref{wprop}.}
\end{proof}

\medskip

 We now address the other component of the concatenation of homotopies. 
\begin{definition}
Given homotopies $f \ra{(f,s,t)} g$ and $g \ra{(g,s',t')} h$, let us put:
\begin{equation}
(t\boxplus t^{\prime })(e)=t(e)+t^{\prime }(e)+(w^{(s,s')} \circ \partial _{1})(e).
\end{equation}
\end{definition}

\begin{remark}
{By Lemma \ref{wzero}, if  $s'=0$ (in the first case) or $s=0$ (in the second case), then:}
\begin{equation}
t\boxplus 0_{t}=t \text{  \ and \  } 0_{t'}\boxplus t'=t'.
\end{equation}%
\end{remark}
\begin{theorem}[Concatenation of  homotopies]\label{conchom} Let $f,g$ and $h$ be 2-crossed module morphisms  $\mathcal{%
A\to A}^{\prime }$, $(s,t)$ be a quadratic $f$-derivation connecting $f$
to $g$, and $(s^{\prime },t^{\prime })$ be a quadratic $g$-derivation connecting $g$
to $h$. Then $(s\boxplus s^{\prime },t\boxplus t^{\prime })$ defines a quadratic $f$-derivation connecting $f$ to $h$.
\end{theorem}

\begin{proof}
That $t\boxplus t^{\prime }$ satisfies the conditions of Definition \ref{qder} follows from equation \eqref{xsimp}. That 
$(s\boxplus s^{\prime },t\boxplus t^{\prime })$ connects $f$ to $h$ is analyzed dimensionwise; see \eqref{sourcetarget}. For $n=0$, follows from definition of $s\boxplus s^{\prime }$  and the fact that $B$ generates $R$. For $n=1$ follows from equation \eqref{wprop}, and, for $n=2$,  from  $w^{(s,s')}\circ (\d_1 \circ \d_2)=0$ {(as $(\d_1 \circ \d_2)=0$).} 
\end{proof}

\subsubsection{The groupoid inverse of a quadratic derivation}\label{inverse}
 
{Let $f=(f_2,f_1,f_0)$ and $g=(g_2,g_1,g_0)$ be 2-crossed module morphisms $\mathcal{%
A\to A}^{\prime }$. Let $(s,t)$ be a quadratic $f$-derivation connecting $f$
to $g$. As before, we take $\A=(L,E,R,\partial _{1},\partial _{2},\{,\})$ to be free up to order one, with a chosen basis $B$ of $R$.}
We now define a quadratic $g$-derivation $(\overline{s},\overline{t})$, connecting $g$ to $f$, called the groupoid inverse of $(s,t)$.
We know that: \begin{align*}
g_{0}(r) = f_{0}(r)+(\partial _{1}^{\prime }\circ s)(r)  \text{\quad or, equally, \quad}  f_{0}(r) = g_{0}(r)+(\partial	 _{1}^{\prime }\circ -s)(r).
\end{align*}
Thus $-s\colon R\to E'$ has the right target. However $-s$ is not necessarily a $g_{0}$-derivation.
\begin{definition}
Let $\bar{s}\colon R \to E$ be the unique $g_{0}$-derivation (Lemma \ref{uniquelyextended}) extending  the restriction of the  function $-s$ to $B$; see the diagram below:
\begin{equation}
\xymatrix@R=30pt@C=30pt{\free{B}{R}{\A'_1}{(f_{0},-s)}{(f_{0},\bar{s})}}
\end{equation}%
\end{definition}

\begin{lemma}
We have that $s\boxplus \bar{s}=0_{s}\text{  \ and \  }\bar{s}\boxplus s=0_{s}$.
\end{lemma}

\begin{proof}
{Follows from Lemma  \ref{uniquelyextended} (since the equality holds in the free basis $B \subset R$).}
\end{proof}

%
%
%
%
%

\begin{definition}{
Put  $\bar{t}=-t-(w^{(s,\bar{s})} \circ \partial _{1})$. (Clearly $t\boxplus \bar{t}=0_{t}\text{  \ and \  }\bar{t}\boxplus t=0_{t}$.)}
\end{definition}

\begin{theorem}
If $(s,t)$ is a quadratic $f$-derivation connecting $f$ to $g$, then $(\bar{s},\bar{t})$ is a quadratic $g$-derivation connecting $g$ to $f$.
\end{theorem}

\begin{proof}
Identical to the proof of Theorem \ref{conchom}.
\end{proof}

\subsubsection{The concatenation of homotopies is associative}

Let $f,g,h$ and $k$ be 2-crossed module morphisms $\mathcal{%
A\to A}^{\prime }$. Let: 
\begin{itemize}
\item $(s,t)$ be a quadratic $f$-derivation connecting $f$ to $g$: that is $f \ra{(f,s,t)} g$, 
\item $(s^{\prime },t^{\prime })$ be a quadratic $g$-derivation connecting $g$ to $h$: that is $g \ra{(g,s',t')} h$,
\item $(s^{\prime\prime },t^{\prime\prime })$ be a quadratic $h$-derivation connecting $h$ to $k$: that is $h \ra{(h,s'',t'')} k$.
\end{itemize}
If we choose an element $b\in B$, the free basis of $R$, we clearly have $
(s\boxplus (s^{\prime }\boxplus s^{\prime \prime}))(b)=((s\boxplus s^{\prime })\boxplus s^{\prime \prime })(b)
.$
Therefore, from Lemma \ref{uniquelyextended}, we can conclude that, for all $r \in R$:
\begin{equation}\label{sassociative}
(s\boxplus (s^{\prime }\boxplus s^{\prime \prime}))(r)=((s\boxplus s^{\prime })\boxplus s^{\prime \prime })(r)
. 
\end{equation}

Let us now prove that the concatenation of quadratic derivations is associative, also at the level of their second components. Consider the set map:
\begin{equation}
\begin{tabular}{llll}
$\lambda :$ & $B$ & $\longrightarrow $ & $\left( (R^{\prime }\ltimes
_{\blacktriangleright }E^{\prime })\ltimes _{\blacktriangleright ^{\bullet
}}(E^{\prime }\ltimes _{\blacktriangleright ^{\prime }}L^{\prime })\right)
\ltimes _{\blacktriangleright ^{\dagger }}\left( (E^{\prime }\ltimes
_{\blacktriangleright ^{\prime }}L^{\prime })\ltimes _{\blacktriangleright
_{\ast }}L^{\prime }\right) =\A'_3$ \\
& $b$ & $\longmapsto $ & $(f_{0}(b),s(b),s^{\prime }(b),0,s^{\prime \prime
}(b),0,0)$%
\end{tabular}.
\end{equation}(For notation see Subsection \ref{tet}).
In simplicial notation, if $b \in B$, we have:
\begin{equation*}\hskip-0.1cm \xymatrix{\\ \lambda(b)=}\hskip-0.4cm 
\xymatrix@R=15pt@C=30pt{\tetraprime{f_{0}(b)}{s(b)}{s^{\prime }(b)}{0}{s^{\prime \prime}(b)}{0}{0}}
\end{equation*}%
\noindent (where $f=(f_2,f_1,f_0)$), or, what is the same: %
\begin{equation} 
\xymatrix{\\ \lambda(b)= \quad \quad}
\xymatrix@R=15pt@C=40pt{
& & \s{k_{0}(b)} & & \\
& & & & \\
& & \s{h_{0}(b)} \ar[uu]|{\s{s^{\prime\prime}(b)}} & & \\
& &  & & \\
\s{f_{0}(b)} \ar[rrrr]_{\s{s(b)}}^{\ovalbox{$\s{0}$}} \ar[rruu]_{\s{s(b)+s^{\prime}(b)}} \ar[rruuuu]^{\s{s(b)+s^{\prime}(b)+s^{\prime\prime}(b)}}_{\ovalbox{$\s{0}$}} & & & & \s{g_{0}(b)} \ar[lluu]^{\s{s^{\prime}(b)}} \ar[lluuuu]_{\s{s^{\prime}(b)+s^{\prime\prime}(b)}}^{\ovalbox{$\s{0}$}}
}
\end{equation}
Since $R$ is the free algebra on $B$, there exists a unique algebra homomorphism $Z \colon R \to \A'_3$,
 extending $\lambda\colon B \to \A_3'$; see the diagram below:
\begin{equation}
\xymatrix@R=30pt@C=30pt{\free{B}{R}{\A'_3}{\lambda}{Z}}
\end{equation}%
 {By using the maps $d_{0,1,2,3}$ of Remark \ref{tetfaces} and Lemma \ref{uniquelyextended}, we have}, for each $r \in R$:
\begin{equation}\label{defZ}
\xymatrix{\\ Z(r)=\\}
\xymatrix@R=20pt@C=50pt{
& & \s{k_{0}(r)} & & \\
& & & & \\
& & & & \\
& & \s{h_{0}(r)} \ar[uuu]|{\s{s^{\prime\prime}(r)}} & & \\
\s{f_{0}(r)} \ar[rrrr]_{\s{s(r)}}^{\ovalbox{$\s{w^{(s,s^{\prime })}(r)}$}} \ar[rru]_{\s{(s\boxplus s^{\prime})(r)}} \ar[rruuuu]^{\s{{(s\boxplus s^{\prime}
\boxplus s^{\prime\prime})(r)}}}_{\ovalbox{$\s{w^{(s\boxplus s^{\prime },s^{\prime \prime })}(r)}$}} & & & & \s{g_{0}(r)} \ar[llu]^{\s{s^{\prime}(r)}} \ar[lluuuu]_{\s{(s^{\prime}\boxplus s^{\prime\prime})(r)}}^{\ovalbox{$\s{w^{(s^{\prime },s^{\prime \prime })}(r)}$}}
}
\end{equation}
Similarly to the discussion in \ref{concatenation}, we conclude, by passing to the tuple notation:
\begin{theorem}
There exists an algebra homomorphism:
\begin{equation}
\begin{tabular}{llll}
$Z:$ & $R$ & $\longrightarrow $ & $\left( (R^{\prime }\ltimes
_{\blacktriangleright }E^{\prime })\ltimes _{\blacktriangleright ^{\bullet
}}(E^{\prime }\ltimes _{\blacktriangleright ^{\prime }}L^{\prime })\right)
\ltimes _{\blacktriangleright ^{\dagger }}\left( (E^{\prime }\ltimes
_{\blacktriangleright ^{\prime }}L^{\prime })\ltimes _{\blacktriangleright
_{\ast }}L^{\prime }\right) $,
\end{tabular}
\end{equation}
which has the form:
\begin{equation*}\begin{tabular}{llll}
& $r$ & $\stackrel{Z}{\longmapsto} $ & $\Big(f_{0}(r),s(r),s^{\prime }(r)-(\partial_{2}^{\prime} \circ
w^{(s^{\prime },s^{\prime \prime })})(r),w^{(s^{\prime },s^{\prime \prime
})}(r),$ \\
&  &  & $ 	 \quad s^{\prime \prime }(r)-(\partial_{2}^{\prime} \circ w^{(s\boxplus s^{\prime
},s^{\prime \prime })})(r),w^{(s\boxplus s^{\prime },s^{\prime \prime })}(r)-w^{(s^{\prime
},s^{\prime \prime })}(r),w^{(s^{\prime },s^{\prime \prime })}(r)\Big)$.
\end{tabular}%
\end{equation*}%
\end{theorem}

Let us now put $W=d_{1}\circ Z\colon R \to \A_2$, which gives us the back surface of the tetrahedron in \eqref{defZ}; see Remark \ref{tetfaces}. Given $r \in R$, then $W(r)$ is the triangle:
\begin{equation}\xymatrix{\\ W(r)=\\ }
\xymatrix@R=60pt@C=80pt{\triangln{f_{0}(r)}{g_{0}(r)}{k_{0}(r)}{s(r)}{(s^{\prime }\boxplus s^{\prime \prime})(r)}{(s\boxplus s^{\prime}
\boxplus s^{\prime\prime})(r)}{w^{(s,s^{\prime })}(r)-w^{(s^{\prime },s^{\prime \prime })}(r)+w^{(s\boxplus
s^{\prime },s^{\prime \prime })}(r)}}
\end{equation}%
This 2-simplex has the following form, {for each free generator $b \in B$:}
\begin{equation}\xymatrix{\\ W(b)=\\ }
\xymatrix@R=30pt@C=30pt{\triangln{f_{0}(b)}{g_{0}(b)}{k_{0}(b)}{s(b)}{(s^{\prime }+s^{\prime \prime})(b)}{(s+s^{\prime}+s^{\prime\prime})(b)}{0}}
\end{equation}%
Recall now the construction of $X^{(s,s^{\prime })}\colon R \to \A_2$ in \ref{concatenation}. Given that the extension of a map $B \to \A_2$ to an algebra map $R \to \A_2$ is unique, it follows that, for each $r \in R$:
\begin{equation}\xymatrix{\\ W(r)=\\ }
\xymatrix@R=40pt@C=40pt{\triangln{f_{0}(r)}{g_{0}(r)}{k_{0}(r)}{s(r)}{(s^{\prime }\boxplus s^{\prime \prime})(r)}{(s\boxplus s^{\prime}
\boxplus s^{\prime\prime})(r)}{w^{(s,s^{\prime }\boxplus s^{\prime \prime })}(r)}}\xymatrix{ \\= X^{(s,s^{\prime }\boxplus s^{\prime \prime})}(r)\,\,,\\}
\end{equation} given that {$W(b)=X^{(s,s^{\prime }\boxplus s^{\prime \prime})}(b)$ for each $b \in B$; cf. equation \eqref{sassociative}.
Thus,} if $r \in R$:
\begin{equation*}
w^{(s,s^{\prime })}(r)-w^{(s^{\prime },s^{\prime \prime })}(r)+w^{(s\boxplus
s^{\prime },s^{\prime \prime })}(r)=w^{(s,s^{\prime }\boxplus s^{\prime \prime })}(r),
\end{equation*}
or:
\begin{equation}\label{wchange}
w^{(s,s^{\prime })}(r)+w^{(s\boxplus
s^{\prime },s^{\prime \prime })}(r)=w^{(s,s^{\prime }\boxplus s^{\prime \prime })}(r)+w^{(s^{\prime },s^{\prime \prime })}(r).
\end{equation}
%
%

\begin{theorem}
For every element $r\in R$:
\begin{equation}
(t\boxplus t^{\prime })\boxplus t^{\prime \prime }=t\boxplus (t^{\prime
}\boxplus t^{\prime \prime }).
\end{equation}
\end{theorem}

\begin{proof}
By using equation \eqref{wchange}, we have, for all $e \in E$:
\begin{align*}
((t\boxplus t^{\prime })\boxplus t^{\prime \prime })(e) & = %
(t\boxplus t^{\prime })(e)+t^{\prime \prime }(e)+(w^{(s\boxplus s^{\prime
},s^{\prime \prime })}\circ \partial _{1})(e) \\ 
& = t(e)+t^{\prime }(e)+(w^{(s,s^{\prime })}\circ \partial
_{1})(e)+t^{\prime \prime }(e)+(w^{(s\boxplus s^{\prime },s^{\prime \prime
})}\circ \partial _{1})(e) \\ 
& = t(e)+t^{\prime }(e)+w^{(s,s^{\prime })}(\partial _{1}(e))+t^{\prime
\prime }(e)+w^{(s\boxplus s^{\prime },s^{\prime \prime })}(\partial _{1}(e))
\\ 
& = t(e)+t^{\prime }(e)+t^{\prime \prime }(e)+w^{(s^{\prime },s^{\prime
\prime })}(\partial _{1}(e))+w^{(s,s^{\prime }\boxplus s^{\prime \prime
})}(\partial _{1}(e)) \\ 
& = t(e)+t^{\prime }(e)+t^{\prime \prime }(e)+(w^{(s^{\prime
},s^{\prime \prime })}\circ \partial _{1})(e)+(w^{(s,s^{\prime }\boxplus
s^{\prime \prime })}\circ \partial _{1})(e) \\ 
& = t(e)+(t^{\prime }\boxplus t^{\prime \prime })(e)+(w^{(s,s^{\prime
}\boxplus s^{\prime \prime })}\circ \partial _{1})(e) \\ 
& = (t\boxplus (t^{\prime }\boxplus t^{\prime \prime }))(e).
\end{align*}
\end{proof}

\medskip

We have now {finished proving} the main theorem of this paper.
\begin{theorem}
Let $\A$ and $\A'$ be 2-crossed modules, of commutative algebras. {Suppose that  $\A=(L,E,R,\partial _{1},\partial _{2},\{,\})$ is free up to order one, with a chosen free basis $B$ of $R$.} We have a groupoid ${\rm HOM}(\A,\A')$, whose objects are the 2-crossed module maps $\A \to \A'$,  with morphisms being the homotopies between them. The groupoid operations are the concatenations and inverses of homotopies (quadratic derivations) described in \ref{concatenation} and \ref{inverse}.
%
\end{theorem}
 {The compositions (and the inverses) in the groupoid ${\rm HOM}(\A,\A')$, in general, explicitly depend on the chosen free basis $B$ of $R$. However, it immediately follows:}
\begin{theorem}{
Let $\mathcal{A}$ and $\mathcal{A^{\prime}}$ be two arbitrary crossed modules, where $\mathcal{A}$ is
free up to order one.  The relation below between maps $f,g\colon \A \to \A'$ is an equivalence relation:}
\begin{equation}
\text{\textquotedblleft }f\simeq g\Longleftrightarrow \text{there exists a
quadratic }f\text{-derivation }(s,t)\text{ connecting } f \text{ with } g \text{\textquotedblright }.
\end{equation}%
\end{theorem}
%
\bibliographystyle{plain}

\end{document}